\documentclass[a4paper,12pt]{article}
\usepackage{amsmath,amsfonts,amssymb,amsthm}
\usepackage{xspace,graphicx, color}
\newtheorem{theorem}{Theorem}
\newtheorem{lemma}{Lemma}[section]
\newtheorem{corollary}{Corollary}[section]
\newtheorem{remark}{Remark}[section]

\renewcommand{\epsilon}{\varepsilon}
\newcommand{\talpha}{{\tilde \alpha}}
\newcommand{\Ws}{\mathcal{W}_{\mathrm{s}}}
\newcommand{\Wl}{\mathcal{W}_{\mathrm{l}}}
\newcommand{\efd}{EFD\xspace}
\newcommand{\s}{\mathfrak{s}}

\newcommand{\Su}{\mathfrak{A}}

\newcommand{\tG}{\mathcal{G}}
\newcommand{\logn}[1]{\log^{(n)}\left (#1 \right )}
\newcommand{\lo}[2]{\log^{(#1)}\left (#2 \right )}
\newcommand{\en}[2]{\exp^{(#1)}\left (#2 \right )}
\newcommand{\Gi}{g_{\mathrm{I}}}
\newcommand{\uii}{U}
\newcommand{\Gii}{g_{\mathrm{II}}}

\newcommand{\fX}{\mathcal{Y}}
\newcommand{\wX}{\widetilde X}
\newcommand{\Xs}{X^\mathrm{S}}
\newcommand{\Xd}{X^\mathrm{D}}
\newcommand{\wS}{S^\mathrm{D}}
\newcommand{\Ss}{S^\mathrm{S}}
\newcommand{\ds}{\sigma^\mathrm{D}}
\newcommand{\mkd}{N_k^\mathrm{D}}

\renewcommand{\P}{\mathbb{P}}
\newcommand{\X}{\mathcal{X}}
\newcommand{\fm}{FMM\xspace}
\newcommand{\dfm}{DFMM\xspace}
\newcommand{\sfm}{SFMM\xspace}
\newcommand{\mm}{MMM\xspace}
\begin{document}

\title{Branching with selection and mutation II:\\ Mutant fitness of Gumbel~type}

\author{Su-Chan Park, Joachim Krug, Peter Mörters}

\date{}

\maketitle

\begin{abstract}
We study a model of a branching process subject to \emph{selection}, modeled by giving each family an individual fitness acting as a  branching rate, and \emph{mutation}, modeled by resampling the fitness of a proportion of offspring in each generation. For two large classes of fitness distributions of Gumbel type we determine the growth of the population, almost surely on survival. We then study the empirical fitness distribution in a simplified model, which is numerically indistinguishable from the original model, and show the emergence of a Gaussian travelling wave. 
\end{abstract}

\section{Introduction}

We consider the branching processes with selection and mutation introduced in~\cite{Park2023}. These are 
models of a population evolving in discrete non-overlapping generations with model parameters given by a probability distribution $\mu$ on $(0,\infty)$, which serves as a 
means to sample a random fitness of a mutant, and a mutation probability~$\beta\in(0,1)$. 
For later reference, we denote the tail function by $G(x):= \mu((x,\infty))$.
Note that $G$ is a right-continuous-left-limit function that may be discontinuous.
\medskip

A brief description of the two model variants goes as follows: In each generation a population consists of finitely many individuals each equipped with a positive fitness. Any individual lives only for one generation. Every generation produces a random number of offspring, which is  Poisson distributed with the mean given by the sum over all the fitnesses of the individuals in the generation. Now every offspring individual independently
\begin{itemize}
    \item with probability $1-\beta$ randomly \emph{selects} a parent with a probability proportional to its fitness. The offspring becomes an individual of the next generation with the fitness inherited from the parent;
\item otherwise, with probability $\beta$, it is a mutant and gets a fitness randomly sampled from $\mu$.
\begin{itemize}
    \item In the \emph{fittest mutant model (\fm)} only one mutant with largest fitness among all mutants, if it exists, joins the next generation and the others die immediately.
    \item In the \emph{multiple mutant model (\mm)} all mutants join the next generation.
\end{itemize}    
\end{itemize}
We write $X(t)$ for the number of individuals in generation 
$t$, irrespective of what initial condition is used and which model 
variant is under consideration.
Further 
discussion of the motivation behind this model can be found in the first  paper of this series~\cite{Park2023}. Other branching models including selection or mutation are \cite{AZEVEDO2021108708, 10.1214/12-AIHP504, BG} or \cite{68dcb93a-3ef5-3c08-8997-49fb84f165eb}.
Similar models have been applied for the description of the genetic structure of proliferating tumors and growing populations of pathogens \cite{Durrett2010,Cheek2020,Nicholson2023}\medskip

Our focus in this paper is on the case of unbounded fitness distributions $\mu$ with light tails at infinity, but to put this into context we briefly review known results first on bounded and second on unbounded heavy-tailed random variables.
\medskip

Suppose first that $a:=$esssup $\mu<\infty$ let
$\lambda^*:=(1-\beta)a$. In the MMM, if $\beta\int \frac{a}{a-x} \, \mu(dx)\geq 1$ there is  
a unique $\lambda\geq \lambda^*$ such that
$$1= \int \frac{\beta x}{\lambda-(1-\beta)x} \mu(dx).$$
Then, almost surely on survival, we have
$$\lim_{t\to\infty} \frac{\log X(t)}{t}=\log \lambda.$$
Otherwise,  
and always in the FMM, we have, almost surely on survival, 
$$\lim_{t\to\infty} \frac{\log X(t)}{t}=\log \lambda^*.$$
This is shown in~\cite{Dereich2017} for a continuous-time variant of the model and the proof extends to the MMM. For the FMM note that in generation $t$ there are at most {$t+1$ families with fitness $W_0, \ldots, W_{t}$ present},  each growing at rate $\log ((1-\beta) W_i)$. The overall growth rate is therefore bounded from above by $\log \lambda^*$ and also from below as $\limsup W_t=a$ almost surely on survival.
So, irrespective of the finer details of $\mu$, we see exponential growth of the population.%
\medskip%

In the case of a slowly decreasing tail at infinity, i.e.\ when
the tail function $G$
is regularly varying with index $-\alpha$, for some $\alpha>0$, we have doubly exponential growth. We show in \cite{Park2023} that, for $T$ the unique integer such that 
$$\frac{(T-1)^T}{T^{T-1}} <\alpha \leq 
\frac{T^{T+1}}{(T+1)^{T}},$$
in either MMM or FMM, almost surely on survival,
$$\lim_{t\to\infty} \frac{\log\log X(t)}{t}=\frac1T \log \frac{T}\alpha,$$
i.e.\ we have doubly exponential growth of the population.
The present paper is concerned with unbounded fitness distributions with \emph{light tail at infinity}. In analogy to the classification of distribution as extremal types we denote this class of fitness distributions as Gumbel type
\cite{deHaan06}. The classification of fitness distributions in terms of extreme value classes plays an important role in the theory of evolutionary adaptation \cite{Bataillon2014}. In this context it has been argued that the Gumbel type is the most relevant case biologically \cite{Gillespie1983,Orr2002,Joyce2008}.
\medskip

For unbounded fitness distributions of Gumbel type the population grows at a rate between exponential and doubly exponential. 
This is a wide range that cannot be easily covered by a single functional expression. Therefore we introduce parametrised subclasses of fitness distributions and show how the population grows for these subclasses in dependence of the parameters. Before stating our full results in Section~2 we describe an interesting example to give a flavour.%
\medskip%

We look at fitness distributions with  stretched exponential tail satisfying
$$
\lim_{x\to\infty}
\frac{\log(1/G(x))}{  x^\alpha L(x)}=1,$$
for a slowly varying function $L$ and $\alpha>0$. In this case, for both MMM and FMM we show in Theorem~1 that the population grows like
$$
\lim_{t\to\infty} \frac{\log X(t)}{t \log t}=\frac1\alpha,
\qquad \text{ 
almost surely on survival.}
$$
The superexponential growth is driven by the fitness $W_t$ of the fittest mutant in generation $t$ satisfying
$$
\lim_{t\to\infty}\, \frac{W_t}{t^{1/\alpha}(\log t)^{1/\alpha} (\alpha L(t^{1/\alpha}))^{-1/\alpha}}=1.
$$
In Section~3 we describe the subtle interplay of population size and fittest mutant heuristically in terms of a differential equation. Simulations demonstrated in Section~7 show that the distribution of fitness in a positive proportion of the population in generation $t$ concentrates 
around the value
$$v(t):= \alpha^{-1/\alpha} t^{1/\alpha} L(t^{1/\alpha})$$
in the shape of a Gaussian travelling wave of width $v(t)/\sqrt{\alpha t}$. In Theorem~\ref{Thm:sfm} we prove this phenomenon rigorously for a simplified model where the driving fitness $W_t$ is replaced by its deterministic asymptotics. 
\medskip

The rest of this paper is organised as follows. Full results on the growth of the population and the driving fitness are formulated as Theorem~1 and~2 in Section~2. The section also formulates, as Theorem~3, the conjectured behaviour of the travelling wave for the full model. Section~3 heuristically describes the interplay of these quantities. Section~4 contains preparation for the proofs of Theorem~1, given in Section~5, and Theorem~2, given in Section~6. Section~7 explains the approximations needed to simulate and prove the travelling wave result restated now in rigorous form as Theorem~4. We finish the paper with concluding remarks in Section~8.

\section{\label{Sec:model}Main results} 

In the \fm, $X(t)$ is generally different from the total number of
offspring of all particles in generation~$t-1$. We therefore denote by 
$\Xi(t)$ the total number of offspring of all individuals in 
generation $t-1$, including immediately dead ones, if there are any. 
By $Q_t$ we denote the largest fitness \textit{in the population} in generation
$t \ge 0$ and by $W_t$ 
the largest fitness among all mutants in generation $t \ge 1$. 
Note that $W_t \le Q_t$ and $W_t$ can be strictly smaller than $Q_t$. The number of non-mutated descendants in generation $s\ge t$ of the fittest mutant in generation $t$ will be denoted by $N_t(s)$ with the convention that $N_t(t) = 1$.  For convenience we set 
$N_t(s) = 0, W_t = 0$ if there is no mutant in generation~$t$ and $Q_t = 0$ if $X(t) = 0$. Also set $\Xi(0)=X(0)$, $W_0=Q_0$. 

\subsection{Tail functions}

To classify the decay of the tail function $G$ in a way that allows the description of the growth rates of the population size,
we denote by $\log^{(n)}$ the $n$th iterated logarithm, write $f_1(t) \sim f_2(t)$ to mean that 
the ratio of the two expressions converges to one as $t$ goes to infinity,
and assume
\begin{align}
\label{Eq:gasym}
	\lo{n_1}{1/G(x)}\sim  \big( \lo{n_2}{x}\big)^\talpha L\big( \lo{n_2}{x}\big),
\end{align}
where  $n_1, n_2$ are non-negative integers, $\alpha$ is a positive number, and
$L(x)$ is assumed to satisfy\footnote{If $L$ is a slowly varying function,
then this condition is naturally satisfied.}
\begin{align}
\lim_{x\rightarrow\infty} \frac{L(x)}{x^\epsilon}=
\lim_{x\rightarrow\infty} \frac{1}{L(x)x^\epsilon}=0,
\label{Eq:conomega}
\end{align}
for any $\epsilon>0$. 
Apart from this assumption, henceforth called 
{\bf (A1)}, we use three further technical assumptions on $L$ in \eqref{Eq:gasym}, namely
\begin{description}
\item[(A2)]
If a positive function $\ell$ satisfies \eqref{Eq:conomega}, then 
$\displaystyle
	\lim_{x\rightarrow\infty} \frac{L(x \ell(x))}{L(x)}=1.$
\item[(A3)] $L$ is four-times continuously differentiable, at least for
sufficiently large argument.
\item[(A4)]
$\displaystyle	\lim_{x\rightarrow\infty} \left ( \frac{d}{d\log x} \right )^j 
\log L(x^{\gamma})=0,$
for nonnegative integer $j$ and positive real $\gamma$.
\end{description}
Assumption~{\bf (A2)} will be used in Section~\ref{Sec:XW}. It is a stronger condition
than $L$ being a slowly varying function. Assumption~{\bf (A3)} will be used in Section~\ref{Sec:pfemp}.  Note that even if $G$ is discontinuous, we can, in most cases, find a four-times continuously differentiable $L$.
Assumption~{\bf (A4)} will be used in the proof of Lemma~\ref{Lem:Gx}
and in Section~\ref{Sec:deter}.\medskip

As an example of $L$ satisfying all four assumptions, we consider
\begin{align}
L(x) = \prod_{k=1}^m (\lo k x )^{\gamma_k}
\label{Eq:Lex}
\end{align}
with real $\gamma_k$'s. 
Obviously, \eqref{Eq:Lex} cannot exhaust all functions satisfying the above four assumptions; an example that does not take the form \eqref{Eq:Lex} is $\exp(\sqrt{\log x})$. The proofs of the main theorems apply to any function $L$ that satisfies the above four assumptions.\medskip

In this paper, we are interested in Gumbel type tail functions with unbounded support, meaning that at infinity $G$ decays faster than polynomially, i.e., for any positive $\gamma$,
\begin{align}
\label{Eq:Gfast}
\lim_{x\rightarrow\infty} x^{\gamma} G(x)= 0.
\end{align}
We now figure out\footnote{It is of course possible that $G(x)$ satisfies \eqref{Eq:Gfast} but not \eqref{Eq:gasym}.} for which parameters $n_1$, $n_2$ and $\talpha$, \eqref{Eq:Gfast} holds.
If $n_1 < n_2$, then $G$ satisfies
\begin{align}
\label{Eq:Gslow}
\lim_{x\rightarrow\infty} \frac{x^{-\epsilon}}{G(x)} = 0
\end{align}
for any positive $\epsilon$.  
As this $G$ decays slower than any Fr\'echet type 
tail function,  the long-time evolution is dominated by the
largest fitness alone as in the Fr\'echet type with $\alpha< 0.5$, as studied in~\cite{Park2023}.
If $n_1 > n_2$, then $G$ satisfies \eqref{Eq:Gfast}, which will be our concern. We define the $n$-th iterated exponential function $\exp^{(n)}$ as the inverse of 
$\log^{(n)}$ with the convention $\exp^{(0)}(x)=\lo0x=x$.
In case $n_2>0$, we have a rough bound for sufficiently large $x$ as
\begin{align*}
1/G(x) &\ge \en{n_1}{\lo{n_2}{x}^{\talpha-\epsilon}} \\
&= \en{n_1-n_2}{ \en{n_2}{\lo{n_2}{x}^{\talpha-\epsilon}}} \\
&= \en{n_1-n_2}{ \en{n_2+1}{(\talpha-\epsilon)\lo{n_2+1}{x}}} 
\ge \en{n_1-n_2}{x^{\talpha-\epsilon}},
\end{align*}
where we have used Lemma~\ref{Lem:expnd} for the last inequality,
and
\begin{align}
	1/G(x) \le \en{n_1}{\lo{n_2-1}x} =\en{n_1-n_2+1}{x}.
\end{align}
In this context, limiting ourselves to the case with $n_1>n_2=0$ would give
a guide for $n_1>n_2>0$. For example, inspecting Theorem~\ref{Th:mainthm}
suggests that almost surely on survival
$$
\lim_{t\rightarrow \infty} \frac{\lo{2}{X(t)}}{\log t} = 1
$$
for any case with $n_1>n_2\ge0$.
The remaining case is $n_1=n_2$. 
If $n_1=n_2=0$, then $G$ does not satisfy \eqref{Eq:Gfast}. In fact, this
$G$ becomes a Fr\'echet-type tail function 
already studied in~\cite{Park2023}. 
If $n_1=n_2 >0$, then how fast $G$ decays 
is determined by $\talpha$. If $0<\talpha <1$, then $G$ satisfies \eqref{Eq:Gslow}.
If $\talpha > 1$, then $G$ satisfies \eqref{Eq:Gfast}.
If $\talpha=1$, then how fast $G$ decays depends on the explicit form of
$L$.
For example, assume $L(x) = (\log x)^\gamma \bar L(\log x)$ with $\bar L$ to satisfy
\eqref{Eq:conomega}. If $\gamma > 0$, then $G$ satisfies \eqref{Eq:Gfast},
while if $\gamma < 0$, then $G$ satisfies \eqref{Eq:Gslow}.
If $\gamma=0$, then how fast $G$ decays depends on the explicit form of
$\bar L$. In this sense, it is difficult, if not impossible, to
write all possible tail functions that satisfy \eqref{Eq:Gfast}.
We take a rather special form of $L$ for $\talpha \ge 1$; see~\eqref{Eq:Glogn}.
We only study the case $n_1=n_2=1$, but the case with $n_1 = n_2 > 1$ can be
easily studied using the techniques developed in this paper. \medskip

In this paper, we therefore limit ourselves to two cases. The first
case that corresponds to $n_2=0$ and $n_1=n \ge 1$ with $\talpha=\alpha>0$ is 
\begin{align}
	\logn{1/G(x)} \sim \Gi(x) := x^\alpha L(x).
\label{Eq:Gexpn}
\end{align}
The second case that corresponds to $n_1=n_2=1$ with $\talpha\ge 1$ is
\begin{align}
	\frac{\log(1/G(x))}{\log x} \sim \Gii(x):=\Gi(\lo{n} x),
\label{Eq:Glogn}
\end{align}
where $n \ge 1$ and $\alpha > 0$. Note that for the second case
$\talpha = 1 + \alpha>1$ for $n=1$ and  $\talpha=1$ for $n \ge 2$.
From now on, $n$ and $\alpha$ are reserved for this role, with $n$ called the \emph{tail index} and $\alpha$ the \emph{tail parameter}.
When $G$ satisfies \eqref{Eq:Gexpn}, we will say that $G$ is of type I and
when $G$ satisfies \eqref{Eq:Glogn}, we will say that $G$ is of type II.
 Note that not only do the two types of decay not cover the entire Gumbel class, but conversely \eqref{Eq:Gfast} alone cannot guarantee that $G$ falls into the Gumbel class.
For instance, consider $\log G(x) = -x - \sin(x)$, which is of type~I but
does not belong to the Gumbel class (see, e.g., \cite{deHaan06}).

\subsection{Statement of theorems}

Our main concern is how $X(t)$, $W_t$, and the empirical
fitness distribution (\efd) behave at large times $t$ on survival.
The \efd\ is defined via its cumulative distribution function $\Psi(f,t)$~as
\begin{align}
\Psi(f,t):= \frac{1}{X(t)}\sum_{i=1}^{X(t)} \Theta(f-F_i),
\label{Eq:efd_def}
\end{align}
where $F_i$ is the fitness of $i$-th individual and $\Theta(x)$ is the Heaviside
step function with $\Theta(0)=1$.
We denote the mean and the standard deviation of $\Psi(f,t)$ by $S_t$ and
$\sigma_t$, respectively.
In case that no individual is left at $t$, we define
$\Psi(f,t) =1$ for $f\ge 0$ and $S_t=\sigma_t=0$.
We define the survival event $\Su$ and survival probability $p_s$ as
$$
\Su := \{X(t) \neq 0 \text{ for all } t \},\quad
p_s := \P(\Su).
$$
Needless to say, $p_s$ depends on the initial condition, but the initial
condition dependence does not play any role in what follows.
Now we state the main theorems.
\begin{theorem}
\label{Th:mainthm}
If $G$ is of type I,
then almost surely on survival
\begin{align}
\lim_{t\to\infty} \frac{\log X(t)}{t \logn t}=\frac1\alpha,\quad
\lim_{t\to\infty} \frac{W_t}{u_n(t)}=1,
\label{Eq:ft}
\end{align}
where
\begin{align}
u_n(t) &:= 
\left ( \lo{n-1} t \right )^{1/\alpha} \omega_W\left (\lo{n-1} t \right ),
\label{Eq:unt}
\\
\label{Eq:omega}
\omega_W(y)&:= \left (\frac{\log y}{\alpha}\right )^{\delta_{n,1}/\alpha} 
\left [ L \left ( y^{1/\alpha} \right )\right ]^{-1/\alpha} ,
\end{align}
with $\delta_{n,1}$ to be the Kronecker delta symbol.
\end{theorem}

\begin{theorem}\label{Th:mainII}
If $G$ is of type II, 
then almost surely on survival
\begin{align*}
\lim_{t\to\infty} \frac{\lo 2{X(t)}}{\log t}=1+\frac1\alpha,\quad
\lim_{t\to\infty} \frac{\lo 2{W_t}}{\log t}=\frac1\alpha,
\end{align*}
for $n=1$, 
\begin{align*}
\lim_{t\to\infty} \frac{\lo 3 {X(t)}}{\log t}=
\lim_{t\to\infty} \frac{\lo 3 {W_t}}{\log t}=\frac1{1+\alpha},
\end{align*}
for $n=2$, and
\begin{align*}
\lim_{t\to\infty} \frac{1}{\lo{n-1} t}\log\left ( \frac{\lo 2 {X(t)}}{t}\right )=
\lim_{t\to\infty} \frac{1}{\lo{n-1} t}\log\left ( \frac{\lo 2 {W_t}}{t}\right )=-\alpha,
\end{align*}
for $n \ge 3$.
\end{theorem}

Based on simulations, we conjecture the following theorem regarding the \efd\ formulated in the case of the FMM. A rigorously proved version of this result and more details on our simulations will be given in Section~7.

\begin{theorem}[{\bf Conjecture}]
\label{Th:efd}
For each type I tail function,
there are positive functions $v(t)$ and $\s(t)$ such that
$$
\lim_{t\rightarrow\infty} v(t) = \infty,\quad
\lim_{t\rightarrow\infty} \frac{\s(t)}{v(t)} = 0,
$$ and
almost surely on survival
$$
\lim_{t\rightarrow\infty} \Psi(v(t) +  y\s(t),t) = 
\Upsilon(y),$$
where
\begin{align}
\label{Eq:normal}
\Upsilon(y):= \frac{1}{\sqrt{2\pi}}\int_{-\infty}^y \exp \left ( -\frac12x^2
\right ) dx.
\end{align}
In particular, if $n=1$ and $\alpha>2$ or if $n \ge 2$, then
$\s(t) \rightarrow 0$ as $t\rightarrow\infty$ and, for $y\not=0$,
$$
\lim_{t\rightarrow\infty} \Psi(v(t) +  y,t) = \Theta(y) \quad  \text{ almost surely on survival.}
$$
\end{theorem}
\medskip

\begin{remark} \label{Rem:mm}A similar statement is conjectured for the MMM where a fraction $1-\beta$ of the mass in the \efd\ enters the travelling wave and a fraction $\beta$ remains in the bulk.
\end{remark}

\begin{corollary}
Given Theorem~\ref{Th:efd}  the empirical mean fitness satisfies 
$$\lim_{t\to\infty} \frac{S_t}{v(t)}=1 \quad \text{  almost surely on survival.}$$
\end{corollary}

\begin{proof}
Fix $\epsilon>0$. By Markov's inequality, we have
$
1-\Psi( \frac{v(t)}{1+\epsilon},t) \le (1+\epsilon)\frac{S_t}{v(t)}.
$
As $\s(t)/v(t)\rightarrow 0$ as $t\rightarrow\infty$,
Theorem~\ref{Th:efd} implies that 
$
\lim\limits_{t\rightarrow\infty}\Psi( \frac{v(t)}{1+\epsilon},t) = 0.
$
Therefore, almost surely on survival, we have
$
\liminf\limits_{t\rightarrow\infty} \frac{S_t}{v(t)} \ge \frac{1}{1+\epsilon}.
$
As $\epsilon$ was arbitrary we have, almost surely on survival,
$\liminf\limits_{t\rightarrow\infty} \frac{S_t}{v(t)} \ge 1.$
\medskip

Now assume
$
\limsup\limits_{t\rightarrow\infty} \frac{S_t}{v(t)}  > 1.$ Then there is $\epsilon'>0$ and a strictly increasing sequence $(t_k)_{k=1}^\infty$ such that
$
S_{t_k} \ge v(t_k) (1+ 2\epsilon')
$
for all $k$.
As Theorem~\ref{Th:efd} implies
$
\lim\limits_{t\rightarrow\infty} 
\Psi( \frac{1+2\epsilon'}{1+\epsilon'}v(t),t) =1,
$
we have
$
\lim\limits_{k\rightarrow\infty} \Psi( \frac{S_{t_k}}{1+\epsilon'},t_k) =1,
$
which contradicts to the definition of $S_t$. Therefore, we conclude 
that 
almost surely on survival
$
\limsup\limits_{t\rightarrow\infty} \frac{S_t}{v(t)} \le 1,
$
which along with the lower bound gives $S_t \sim v(t)$.
\end{proof}

In Section~\ref{Sec:pfemp}, we will modify our model so that a version of Theorem~\ref{Th:efd} can be proved.

\section{\label{Sec:sim}Heuristic guide to Theorems~\ref{Th:mainthm} and~\ref{Th:mainII}}
Before delving into the proofs, we first
sketch the idea behind Theorems~\ref{Th:mainthm} and \ref{Th:mainII} by a mean-field type analysis of the \mm for a strictly decreasing continuous $G$ with $\Gi(x) \sim x^\alpha$.
Let us assume that at certain time~$t$,  the population size $X(t)$ is very large.
Once $X(t)$ is given, $W_t$ is sampled as 
$
 Z= [ 1- \beta G(W_t)]^{X(t)},
$
where $Z$ is uniformly distributed on $(0,1)$; see Lemma~\ref{Th:Wa}.
Neglecting fluctuation in the sense that 
$-\log G(W_t) \approx \log X(t)$, 
we have
\begin{align}
\log W_t \approx \frac{1}{\alpha} \log^{(n+1)}(X(t))
\label{Eq:Wxe}
\end{align}
for type I and
\begin{align}
\log W_t \approx \begin{cases}
\left [ \log X(t)\right ]^{1/(1+\alpha)}, & n=1,\\
\left [\log^{(n)}(X(t)) \right ]^{-\alpha}\log X(t), & n \ge 2,
\end{cases}
\label{Eq:Wxl}
\end{align}
for type II.
Since the mean fitness $S_t$ is anticipated not to be larger than $W_t$ 
and $\log X(t+1) \approx \log X(t) + \log S_t$, we have
$
\log X(t+1) - \log X(t) \le \log W_t.
$
Treating $t$ as a continuous variable and setting
$y = \log  X(t) $, we assume that 
the solutions of the differential equations
\begin{align}
\frac{dy}{dt} = \frac{1}{\alpha}\logn  y 
\label{Eq:dyI}
\end{align}
for type I
and 
\begin{align}
\label{Eq:dyII}
\frac{dy}{dt} = \begin{cases}
y^{1/(1+\alpha)}, & n=1,\\
y / ( \log^{(n-1)}(y))^\alpha, & n \ge 2,\\
\end{cases}
\end{align}
for type II
give the upper bound for the corresponding $\log X(t)$.
The asymptotic behaviour of the solution of \eqref{Eq:dyI} can be
found as 
$$
\frac{t}{\alpha} = \int^y 
\frac{dx}{\logn x}
= 
\frac{y}{\logn y}
+ \int^y \frac{1}{(\logn x)^2} 
\left ( \prod_{k=1}^{n-1} \frac{1}{\lo k x} \right ) dx
\approx 
\frac{y}{\logn y},
$$
which gives
$$
y \approx \frac{t}{\alpha} \logn y \approx
\frac{t}{\alpha}  \logn t,
$$
where we have used  {\bf (A2)} for $L(x)=\logn x$.
In a similar manner, we find the asymptotic solution of \eqref{Eq:dyII}
as $
y \approx t^{1+1/\alpha}$ if $n=1$,
$
y \approx \exp( t ^{1/(1+\alpha)})$ if $n=2$ and 
$
y \approx
\exp( t ( \log^{(n-2)}(t))^{-\alpha})$
if $n\ge 3$.
Accordingly, we anticipate
$$
\log X(t)\lessapprox \frac{1}{\alpha} t \log^{(n)}(t)
$$
for type I and
$$
\log X(t)\lessapprox 
\begin{cases}t^{1+1/\alpha}, & n=1,\\
\exp \left (t^{1/(1+\alpha)}\right ), & n= 2,\\
\exp \big( t \big( \log^{(n-2)}(t)\big)^{-\alpha}\big), & n \ge 3.
\end{cases}
$$
for type II.
\pagebreak[3]\medskip

Theorems~\ref{Th:mainthm} and \ref{Th:mainII} actually state that to treat the above inequalities as equalities gives a good approximation.
If the inequalities are indeed equalities, then we expect
\begin{align*}
W_t \approx \big( \log^{(n-1)}(t\logn t) \big)^{1/\alpha}
\end{align*}
for type I and
\begin{align*}
W_t \approx 
\begin{cases}
\displaystyle \exp \left (t^{1/\alpha}\right ), & n=1,\\
\displaystyle \exp \left( t^{-\alpha/(1+\alpha)} 
\exp \left (t^{1/(1+\alpha)}\right )\right), & n =2,\\
\displaystyle \exp \Big( \big(\log^{(n-2)}(t) \big)^{-\alpha} 
\exp \big( t  (\log^{(n-2)}(t))^{-\alpha} \big) \Big), & n \ge 3,
\end{cases}
\end{align*}
for type II.
In the following sections, we make the above heuristics rigorous.

\section{\label{Sec:prep}Preparations}

In this section, we collect some tools to be used in the proofs of the
Theorems~\ref{Th:mainthm} and \ref{Th:mainII}.
To be self-contained, we begin by restating Lemma 2 of Ref.~\cite{Park2023} without repeating the proof. \begin{lemma}
\label{Lem:Wt}
On survival, $(W_t)_{t\ge 1}$ is almost surely an unbounded sequence.
\end{lemma}


Other than in Ref.~\cite{Park2023} the gap between the generation where a mutant type first appears and the generation where it may become dominant is unbounded. Therefore we need tight bounds on the Galton-Watson process with Poisson offspring distribution, which become the focus of the rest of this section. We prepare this with some bounds on the Poisson series.

\begin{lemma}
\label{Lem:lowsum}
If $0<b<1$, $\theta>1$, 
$\lfloor b\theta \rfloor \ge 1$, and $(1-b)\theta \ge 1$,
then
\begin{align*}
\sum_{m=0}^{\lfloor b\theta \rfloor} e^{-\theta} \frac{\theta^m}{m!} &\le  \theta e^{-\theta (1-b+b\log b)}.
\end{align*}
\end{lemma}
\begin{proof}
Let $\ell := \lfloor b\theta \rfloor$ and 
$a_m := \theta^m/m!$. 
Note that $\ell \le b\theta  \le \theta-1$ by the assumption. 
Since $a_{m}/a_{m-1} = \theta/m$,
we have 
$a_m \le a_\ell $ for all $m \le \ell < \theta$ and, therefore,
\begin{align*}
\sum_{m=0}^{\ell} \frac{\theta^m}{m!} 
\le  (\ell+1) \frac{\theta^\ell}{\ell!}
\le  \theta \frac{\theta^\ell}{\ell!}.
\end{align*}
Using $m! \ge m^m e^{-m}$ ($m \ge 1$), we find
$
\log \frac{\theta^\ell}{\ell!}
\le  \ell \log \theta - \ell \log \ell + \ell.
$
Observing that $x \log \theta - x\log x + x$ is an increasing function
in the region $0<x<\theta$,  we finally have
$$
\sum_{m=0}^{\ell} e^{-\theta} \frac{\theta^m}{m!} 
\le \theta e^{-\theta + b\theta \log \theta - b\theta \log (b\theta) + b\theta}
= \theta e^{- \theta(1 - b+b \log b)},
$$
as claimed.
\end{proof}
\begin{lemma}
\label{Lem:upsum}
If $B>1$ and $\theta >0$,
then
\begin{align*}
\sum_{k=\lceil B\theta \rceil}^{\infty} e^{-\theta } \frac{\theta ^k}{k!} &\le 
\frac{B}{B-1} e^{-\theta  (1- B+B \log B )}.
\end{align*}
\end{lemma}
\begin{proof}
Let $m := \lceil B\theta  \rceil$. Since $(m+k)! \ge m! m^k$ and $\theta /m\le 1/B<1$, we have
\begin{align*}
\sum_{k=m}^\infty \frac{\theta ^k}{k!}
\le \frac{\theta ^m}{m!} \sum_{k=0}^\infty \left ( \frac{\theta }{m} \right )^k
= \frac{\theta ^m}{m!} \frac{1}{1-(\theta /m)} 
\le \frac{B}{B-1} e^{m \log \theta  - m \log m + m}
\le \frac{B}{B-1} e^{B\theta  - B\theta \log  B},
\end{align*}
where we have used $m! \ge m^m e^{-m}$ and that
$x \log \theta  - x\log x + x$ is a decreasing function
in the region $\theta <x$.
Multiplying by $e^{-\theta }$, we get the desired inequality.
\end{proof}

\noindent{\bf Definition.}
By $(\X_t)_{t\ge 0}$, we mean a classical
Galton-Watson process with Poisson
offspring number distribution with mean $\theta$, starting in generation 0
with a single individual. 

\begin{remark}
\label{Rem:nPoi}
Conditioned on $\X_{t-1} = m$ for a nonnegative integer $m$,
$\X_t$ is a Poisson-distributed random variable with mean $m \theta$.
\end{remark}
\begin{lemma}
\label{Lem:GWlow}
If $0<b<1$, $\theta\ge f \ge 1/(1-b+b\log b)$, and $x \ge 1$, then,
\begin{align}
\label{Eq:nlow}
\P(\X_t \ge  bx f\vert \X_{t-1}\ge x) \ge 1-xf e^{- x f(1-b+ b \log b)}.
\end{align}
\end{lemma}
\begin{proof}
By assumption, $(1-b)f\ge 1$.
If $m \ge x$, Remark~\ref{Rem:nPoi} with Lemma~\ref{Lem:lowsum} gives 
\begin{align*}
\P(\X_t < bxf\vert \X_{t-1} =m)
 \le 
\P(\X_t < bm\theta\vert \X_{t-1} =m)
\le \sum_{k=0}^{\lfloor  bm \theta\rfloor} e^{-m\theta} \frac{(m\theta)^k}{k!}
\le m\theta e^{- m \theta(1-b+b\log b)}.
\end{align*}
Since $z e^{-zc } \le y e^{-y c}$ for all $z\ge y\ge 1/c>0$ and $m\theta  \ge xf\ge 
1/(1-b+b\log b)>0$,
we get
$$
\P(\X_t < bxf\vert \X_{t-1} =m) \le xf e^{-xf(1-b+b\log b)},
$$
which does not depend on $m$ as long as $m\ge x$. 
Now, the proof is completed. 
\end{proof}
\begin{lemma}
\label{Lem:Markov}
Let $A_t := \{ a_t \le \X_t \le b_t \}$, where
$0\le a_t \le b_t-1 \le \infty$ for all $t\ge 0$.
Let $E_{t} := \bigcap_{k=\tau}^t A_k$ for $0 \le \tau < t$. 
Assume $\P(A_\tau) > 0$ and $\P(A_t \vert \X_{t-1}=m ) \ge f_{t}>0$, where 
$m$ is any integer satisfying $a_{t-1} \le m \le b_{t-1}$ and
$f_t$ depends on $a_t$,
$b_t$, $a_{t-1}$, and $b_{t-1}$ but not on $m$.
Then
$$
\P\left ( E_t \right )
\ge \P(A_\tau)\prod_{k=\tau+1}^t f_k .
$$
\end{lemma}
\begin{proof}
For $t=\tau+1$, the proof is trivial. So we assume $t\ge \tau+2$.
Note that
$$
E_t = A_t \cap A_{t-1} \cap E_{t-2} 
= \bigcup_{m=\lceil a_{t-1} \rceil}^{\lfloor b_{t-1} \rfloor} \left (A_t\cap  \{\X_{t-1} = m \} \cap E_{t-2} \right ).  
$$
Using the countable additivity of the probability measure and the Markov property of $\X_t$, 
\begin{align*}
\P(E_t) 
&= \sum_{m=\lceil a_{t-1}\rceil}^{\lfloor b_{t-1} \rfloor}
\P(A_t \vert \X_{t-1}=m) \P\left (  \{\X_{t-1} = m \} \cap E_{t-2} \right )\\
&\ge f_t \sum_{m=\lceil a_{t-1}\rceil}^{\lfloor b_{t-1} \rfloor}
\P\left (\{\X_{t-1} = m \} \cap E_{t-2} \right )
= f_t \P(E_{t-1}).
\end{align*}
Iterating the above inequality, we get the desired inequality.
\end{proof}
\begin{lemma}
\label{Lem:nPoi}
If $0<b<1$, $\theta\ge f \ge 1/(1-b+b\log b)$, and $bf > 1$, then
\begin{align*}
\P\left (\X_t \ge b^t f^t \text{ for all } t \ge \tau \vert \X_\tau \ge b^\tau f^\tau \right )
\ge 1 - f \big( 1 + \tfrac{1}{\log(bf)}\big) e^{-f (1-b+b\log b)},
\end{align*}
for any nonnegative integer $\tau$.
Note that the right hand side 
does not depend on $\tau$.
\end{lemma}
\begin{proof}
For any event $E$, we write
$
\P_c(E):= \P(E\vert \X_\tau \ge b^\tau f^\tau)
$
in this proof.
Define 
$$
A_t := \{ \X_t \ge b^t f^t\},\quad
C_t := \bigcap_{k=\tau+1}^t A_k,\quad 
C := \bigcap_{k=\tau+1}^\infty A_k.
$$
Note that 
$$
\P_c(A_\tau)=1,\quad
\P_c\left (\X_t \ge b^t f^t \text{ for all } t \ge \tau \right )=\P_c(C)
=\lim_{t\rightarrow\infty} \P_c(C_t)  .$$
Using \eqref{Eq:nlow} with $ x \mapsto (bf)^{t-1}$, we have
\begin{align*}
\P(A_t \vert  A_{t-1} ) \ge 1 -  
f (bf)^{t-1} \exp \left [ - f (bf)^{t-1}(1-b+b\log b) \right ]
=: 1-d_t.
\end{align*}
By Lemma~\ref{Lem:Markov}  we can write
$$
\P_c(C)=\lim_{t\rightarrow\infty} \P_c(C_t) 
\ge \prod_{t=\tau+1}^\infty (1-d_t)
\ge 1 - \sum_{t=\tau+1}^\infty d_t
\ge 1 -  \sum_{t=1}^\infty d_t.
$$
Since 
$(c^{t-1} \exp(-a c^{t-1}))_{t\ge 1}$ is a decreasing 
sequence for $a \ge 1$ and $c>1$, we have
\begin{align}
\nonumber
&\sum_{t=1}^\infty c^{t-1} \exp \left (-a c^{t-1}\right ) 
=  e^{-a} 
+\sum_{t=2}^\infty c^{t-1}\exp \left (-a c^{t-1}\right ) \\
&\le e^{-a} + \int_1^\infty c^{t-1} \exp \left (-a c^{t-1}\right ) dt
= \left ( 1 + \frac{1}{a \log c} \right )e^{-a}
\le \left ( 1 + \frac{1}{\log c} \right )e^{-a}.
\label{Eq:sumxt}
\end{align}
Plugging $c=bf$ and $a=f(1-b+b\log b)$ into \eqref{Eq:sumxt}, we have the desired result.
\end{proof}
\begin{remark}
\label{Rem:bc}
If we restrict the condition of parameters in Lemma~\ref{Lem:nPoi}
to be $bf\ge e$ and $0<b\le b_c<1/2$, where 
$b_c$ satisfies $1-b_c+b_c \log b_c = 1/2$, 
then we can use
\begin{align}
\label{Eq:Plow}
\P\left (\X_t \ge b^t f^t \text{ for all } t \ge \tau \vert \X_\tau \ge b^\tau f^\tau \right )
\ge 1 - 2f e^{-f /2}.
\end{align}
\end{remark}
\begin{lemma}
\label{Lem:GWup}
If $B>1$, $f\ge \theta>0$, and $x \ge 1$, then
\begin{align*}
\P(\X_t \le  Bxf\vert \X_{t-1}\le x) 
\ge 1-\frac{B}{B-1} e^{ - xf B(\log B - 1) }.
\end{align*}
\end{lemma}
\begin{proof}
Set $\X_{t-1}=m$. If $m=0$, the above inequality is trivially true.
So we only consider $1 \le m \le x$.
Let $B' := B xf/(m\theta)\ge B$. Then, Remark~\ref{Rem:nPoi} 
together with Lemma~\ref{Lem:upsum}
gives
\begin{align*}
\P(\X_t >  Bx f\vert \X_{t-1}=m) 
= \P(\X_t > B'm\theta \vert \X_{t-1}=m)
\le \sum_{k=\lceil B'm\theta \rceil}^{\infty} \frac{(m\theta)^k}{k!}
\le \frac{B'}{B'-1} e^{ - m\theta B' (\log B' - 1) },
\end{align*}
where we have used $e^{-y} \le 1$ for $y \ge 0$.
Since $xfB = m\theta B'$, $\log B'\ge \log B$, and $y/(y-1)$ is a decreasing function of $y>1$, 
we have
\begin{align*}
\P(\X_t >  Bx f\vert \X_{t-1}=m) 
\le \frac{B}{B-1} e^{ - xfB (\log B - 1) },
\end{align*}
which is valid for any $m \le x$.
Now the proof is completed.
\end{proof}
\begin{remark}
In case $B\ge e^2>2$, we can use 
\begin{align}
\P(\X_t \le  Bxf\vert \X_{t-1}\le x) 
\ge 1-2 e^{-xf}.
\label{Eq:Pup}
\end{align}
\end{remark}

We next describe the distribution of $W_t$, conditioned on $\Xi(t) = N$.
\begin{lemma}
\label{Th:Wa}
For any $x \geq 0$,
\begin{align*}
\P(W_{t} \leq x \vert  \Xi(t) = N)
=\left (1-\beta G(x) \right )^{N}.
\end{align*}
\end{lemma}
\begin{proof}
First fix a positive integer $m$ and by $W_t^{_{(m)}}$ is denoted the largest of $m$ independently
sampled fitnesses with convention  $W^{_{(0)}}_t=0$. 
Then
$\P(W^{_{(m)}}_{t} \leq  x) = (1-G(x))^m.$
Let $q_m$ be the probability that $m$ mutants arise out of $N$.
Then,
\begin{align*}
\P(W_t\le x\vert \Xi(t) = N)
& = \sum_{m=0}^N \P(W_t^{_{(m)}}\le x)q_m
 = \sum_{m=0}^N (1- G(x))^m  q_m \\ & = 
\sum_{m=0}^N (1- G(x))^m
\binom{N}{m} \beta^m (1-\beta)^{N-m}
= (1-\beta G(x))^{N},
\end{align*}
as claimed.
\end{proof}
\begin{remark}
In case $X(t) \le \Xi(t)\le y$, we will use the inequality 
\begin{align}
\label{Eq:Wup}
\P(W_t \le x \vert  \Xi(t) \le y) \ge 1 - \beta y G(x),
\end{align}
where we have used $(1-z)^m \ge 1 - mz$ for $0\le z\le 1$ and $m \ge 1$.
In case $\Xi(t)\ge X(t) \ge y\ge 0$, we will use the inequality 
\begin{align}
\label{Eq:Wdown}
\P(W_t \ge x \vert  X(t) \ge y) \ge 1-e^{-\beta yG(x)},
\end{align}
where we have used $e^{-yz} \ge (1- z)^y$ for $0\le z \le 1$.
\end{remark}
\section{\label{Sec:XW} Proof of Theorem~\ref{Th:mainthm}}

We first provide a heuristic argument for a more accurate estimate
of $W_t$ than in Section~\ref{Sec:sim}.
As in Theorem~\ref{Th:mainthm} we assume
$$
\log X(t) \approx \frac{t}{\alpha} \logn t .
$$ 
Then, we approximate
$$
\logn{X(t)} \approx \lo{n-1}t \left ( \frac{\log t}{\alpha} \right )^{\delta_{n,1}}.
$$
Now using the mean-field type approximation
$ g_I(W_t) \approx \logn {1/G(W_t)} \approx \logn{X(t)}$, 
we get an approximate $W_t$ by a solution $x$ of the equation
$$
 x^\alpha L(x)= g_I(x) = 
\lo{n-1}t \left (\frac{\log t}{\alpha}\right )^{\delta_{n,1}}.
$$
We can find an approximate solution of the above equation as
\begin{align*}
x &= L(x)^{-1/\alpha} 
\left ( \lo{n-1} t \right )^{1/\alpha} \left (\frac{\log t}{\alpha}\right )^{\delta_{n,1}/\alpha} 
\\
&\approx
\big( \lo{n-1} t \big)^{1/\alpha} \left (\frac{\log t}{\alpha}\right )^{\delta_{n,1}/\alpha} 
\left [ L \big( ( \lo{n-1} t )^{1/\alpha} \big)\right ]^{-1/\alpha} = u_n(t).
%
\end{align*}
where we have used {\bf (A2)}. 
By construction, we have
\begin{align}
g_I(u_n(t))\sim \lo{n-1}t \left (\frac{\log t}{\alpha}\right )^{\delta_{n,1}},
\label{Eq:gunt}
\end{align}
which will play an important role in proving Theorem~\ref{Th:mainthm}.
In the proof, no distinction between \mm and \fm is necessary.
For the proof, we begin with estimating $G(W_t)$ using an inequality relating the iterated exponential function $\en{n}{x}$ and 
the iterated logarithm $\logn{x}$.

\begin{lemma}
\label{Lem:expnd}
For any positive integer $n$ and for any positive $x$, 
\begin{align}
\en{n}{x \logn{t}} \ge t^{x},
\label{Eq:enln}
\end{align}
as long as $\logn{t}>1$.
\end{lemma}
\begin{proof}
For $n =1$, the inequality is trivially valid as an equality.
Now assume that the inequality is satisfied for $n=\ell$.
Consider $t$ with $\lo{\ell+1}{t} > 1$, which gives $\lo{\ell}{t} > e > 1$ and $t>1$.
Abbreviate $y := (\lo{\ell}{t})^{x-1}$ for $x>0$. Since $\lo{\ell}{t}>e$, we have $y>e^{x-1}>x$.
By assumption, we have
\begin{align*}
\en{\ell+1}{x\lo{\ell+1}{t}}
=\en{\ell}{\left (\lo{\ell}{t}\right )^{x}}
= \en{\ell}{y\lo{\ell}t} \ge t^{y} \ge t^{x},
\end{align*}
so that the claimed inequality is also valid for $n=\ell+1$.
Induction completes the proof.
\end{proof}
\begin{lemma}
\label{Lem:Gx}
Fix $\epsilon$ such that
$0< \epsilon < \frac{2\alpha}{3\alpha+8 + \sqrt{\alpha^2 + 48 \alpha + 64}},$
and let
$\epsilon_1 := \frac{4 \epsilon}{\alpha(1-2\epsilon)(1-\epsilon) - 4 \epsilon}$.
Then there is $\tau_0$ (depending on $n$) such that, for all $t \ge \tau_0$,
\begin{align*}
\log G((1-\epsilon_1)u_n(t)) \ge - \frac{1-2\epsilon}{\alpha}
t \logn t ,\\
\log G((1+\epsilon_1)u_n(t)) \le - \frac{1+2\epsilon}{\alpha}
t \logn t.
\end{align*}
\end{lemma}
\begin{proof}
First note that $\epsilon<1/2$ for any $\alpha>0$, $0<\epsilon_1 < 1$,
and
\begin{align}
(1-2 \epsilon)\left [ 1 + \alpha (1-\epsilon) \tfrac{\epsilon_1}{1+\epsilon_1}
\right ]= 1 + 2\epsilon,\quad
\tfrac{1+2\epsilon}{1+\alpha \epsilon_1(1-\epsilon)}
\le 1-2\epsilon.
\label{Eq:ee1}
\end{align}
By \eqref{Eq:gunt}, 
there is $\tau_1$  such that 
\begin{align}
\tfrac{1-2 \epsilon}{1-\epsilon} \lo{n-1}t \left ( \tfrac{\log t}{\alpha}
\right )^{\delta_{n,1}} 
\le g_I(u_n(t))
\le
\tfrac{1+ 2\epsilon}{1+\epsilon}\lo{n-1} t \left ( \tfrac{\log t}{\alpha}
\right )^{\delta_{n,1}},
\label{Eq:uncon}
\end{align}
for all $t \ge \tau_1$.
Now we show that there is $\tau_2$ such that, for all $t\ge \tau_2$,
\begin{align}
\nonumber
\lo{n-1}{\tfrac{1+ 2\epsilon}{\alpha}\, t \logn t} \le (1+2\epsilon) \lo{n-1} t
\left ( \tfrac{\log t}{\alpha} \right )^{\delta_{n,1}},\\
\lo{n-1}{\tfrac{1- 2\epsilon}{\alpha} \, t \logn t} \ge (1-2\epsilon) \lo{n-1} t
\left ( \tfrac{\log t}{\alpha} \right )^{\delta_{n,1}}.
\label{Eq:tlogt}
\end{align}
For $n=1$, this is obvious.
For $n\ge 2$, existence of $\tau_2$ follows from
$$
\lim_{t\rightarrow\infty} \tfrac{1}{\lo{n-1}t}\lo{n-1}{\tfrac{1\pm 2\epsilon}{\alpha} \, t \logn t} = 1.
$$
By the mean value theorem, there is $\epsilon_\pm$ such that
$0\le \epsilon_\pm \le \epsilon_1$ and
\begin{align}
\label{Eq:mv}
g_I((1\pm\epsilon_1)u_n(t)) = g_I(u_n(t)) \pm
\epsilon_1 u_n(t) \left .\frac{dg_I}{dx} \right \vert_{x=(1\pm \epsilon_\pm) u_n(t)}.
\end{align}
We do not make the $t$-dependence of $\epsilon_\pm$ explicit, as in the following we will only use the inequality $0 < \epsilon_\pm < \epsilon_1$. By {\bf (A4)} with $j=\gamma=1$ and \eqref{Eq:Gexpn},
\begin{align*}
 \alpha=
\lim_{x\rightarrow\infty} \frac{x}{g_I(x)}\frac{dg_I(x)}{dx} 
=\lim_{x\rightarrow\infty} \frac{d \log g_I(x)}{d \log x}. 
\end{align*}
Hence
there is $x_1$ such that,  for all $x'\ge x\ge x_1$, we have 
$g_I(x') \ge g_I(x)$ and
\begin{align}
\label{Eq:dgdx}
\alpha(1-\epsilon) \frac{g_I(x)}{x} \le \frac{dg_I}{dx},
\end{align}
and 
\begin{align}
\exp\left ( - \en{n-1}{ (1+\epsilon) g_I(x)} \right )
\le G(x) \le
\exp\left ( - \en{n-1}{ (1-\epsilon) g_I(x)} \right ).
\label{Eq:Gineq}
\end{align}
Hence, if $(1-\epsilon_1)u_n(t) > x_1$, then we have
\begin{align*}
g_I((1+\epsilon_1)u_n(t)) 
&\ge g_I(u_n(t)) + \epsilon_1 u_n(t) \alpha (1-\epsilon)
\frac{g_I((1+\epsilon_+)u_n(t))}{(1+\epsilon_+) u_n(t)}\\
&\ge g_I(u_n(t)) + \frac{\epsilon_1}{1+\epsilon_1}  \alpha (1-\epsilon)
g_I((1+\epsilon_+)u_n(t))\\
&\ge g_I(u_n(t)) \left [ 1 + \alpha (1-\epsilon) \frac{\epsilon_1}{1+\epsilon_1}
\right ],
\end{align*}
where we have used \eqref{Eq:mv}, \eqref{Eq:dgdx}, $1/(1+\epsilon_+) \ge 1/(1+\epsilon_1)$ and
$g_I((1+\epsilon_+)u_n(t)) \ge g_I(u_n(t))$;
and 
\begin{align}
\nonumber
g_I((1-\epsilon_1)u_n(t)) &\le g_I(u_n(t)) - \epsilon_1 u_n(t) \alpha (1-\epsilon)
\frac{g_I((1-\epsilon_-)u_n(t))}{(1-\epsilon_-) u_n(t)}\\
&\le g_I(u_n(t)) - \epsilon_1  \alpha (1-\epsilon) g_I((1-\epsilon_1)u_n(t)),
\label{Eq:g-}
\end{align}
where we have used \eqref{Eq:mv}, \eqref{Eq:dgdx}, $-1/(1-\epsilon_-)\le -1$
and $-g_I((1-\epsilon_-)u_n(t)) \le -g_I((1-\epsilon_1)u_n(t))$.
We can rewrite \eqref{Eq:g-} as
$g_I((1-\epsilon_1)u_n(t)) \le \left[ 1+\alpha \epsilon_1(1-\epsilon)\right ]^{-1}
g_I(u_n(t))$.
Therefore, there is $\tau_3$ such that for all $t \ge \tau_3$ we have
\begin{align}
\nonumber
&(1-\epsilon)  g_I((1+\epsilon_1)u_n(t))
\ge (1-\epsilon)  g_I(u_n(t)) 
\left [ 1 + \alpha (1-\epsilon) \frac{\epsilon_1}{1+\epsilon_1} \right ]\\
\nonumber
&\ge (1-2\epsilon) 
\left [ 1 + \alpha (1-\epsilon) \frac{\epsilon_1}{1+\epsilon_1} \right ]
\lo{n-1}t \left ( \frac{\log t}{\alpha} \right )^{\delta_{n,1}}\\
&=(1+2\epsilon) \lo{n-1}t \left ( \frac{\log t}{\alpha} \right )^{\delta_{n,1}}
\ge\lo{n-1}{\tfrac{1+ 2\epsilon}{\alpha} t \logn t},
\label{Eq:ineq1}
\end{align}
and 
\begin{align}
\nonumber
&(1+\epsilon)  g_I((1-\epsilon_1)u_n(t))
\le (1+\epsilon)  g_I(u_n(t)) 
\left [ 1 + \alpha \epsilon_1(1-\epsilon) \right ]^{-1}\\
\nonumber
&\le \frac{1+2\epsilon}{ 1 + \alpha \epsilon_1(1-\epsilon)}
\lo{n-1}t \left ( \frac{\log t}{\alpha} \right )^{\delta_{n,1}}\\
&\le (1-2\epsilon) \lo{n-1}t \left ( \frac{\log t}{\alpha} \right )^{\delta_{n,1}}
\le\lo{n-1}{\tfrac{1- 2\epsilon}{\alpha} t \logn t},
\label{Eq:ineq2}
\end{align}
where we have used \eqref{Eq:ee1}, \eqref{Eq:uncon}, and \eqref{Eq:tlogt}.
To sum up, there is $\tau_0$ such that for all $t \ge \tau_0$,
\begin{align*}
\log G((1+\epsilon_1)u_n(t)) \le - \tfrac{1+2\epsilon}{\alpha}
t \logn t,
\end{align*}
where we have used \eqref{Eq:Gineq} and \eqref{Eq:ineq1}; and
\begin{align*}
\log G((1-\epsilon_1)u_n(t)) \ge - \tfrac{1-2\epsilon}{\alpha}
t \logn t ,
\end{align*}
where we have used \eqref{Eq:Gineq} and \eqref{Eq:ineq2}.
Now, the proof is completed.
\end{proof}

\begin{lemma}
\label{Lem:Xup}
Assume $X(0) < \infty$ and $Q_0 < \infty$.
Fix $\epsilon$ and $\epsilon_1$ as in Lemma~\ref{Lem:Gx} and let, for~$t\ge 0$, \begin{align*}
A_t :=\left \{\log \Xi(t) \le \tfrac{1+\epsilon}{\alpha} (t+m) \logn {t+m} \right \},\quad
E_t :=\left  \{ W_t \le  \left ( 1 + \epsilon_1\right )u_n(t+m)\right \},
\end{align*}
where $m$ is assumed large enough for the definition to make sense.
We use the convention $\log 0 = -\infty$ throughout the paper.
We define a sequence of events $(D_t)_{t\ge 0}$ iteratively as
$$
D_0 = A_0 \cap E_0,\quad D_t = A_t \cap E_t \cap D_{t-1}. 
$$
Let $D := \bigcap_{t=0}^\infty D_t$.
Then,
$$
\lim_{m\rightarrow\infty} \P\left ( D \right ) = 1.
$$
\end{lemma}
\begin{proof}
%
Since $\liminf_{t\rightarrow\infty}u_n(t)=\infty$ and $\logn t$ is an unbounded and
increasing function, there is $t_1$ such that 
$u_n(m) \ge 1$, $(1+\epsilon) m \logn m \ge \alpha (m+1)$, 
$(1+\epsilon) m \logn m\ge \alpha \log X(0) $
and $(1+\epsilon_1)u_n(m) \ge Q_0$ for all $m > t_1$.
Let 
\begin{align*}
H(x)&:= 
\frac{1+\epsilon}{\alpha} (x+1) \logn{x+1} - 
\frac{1+\epsilon}{\alpha} x \logn{x} - 
\log \left ( u_n(t)\right ) - \log(1+\epsilon_1) \\
&= \frac{x+1}{\alpha'} \left [ \logn{x+1}-\logn{x} \right ] + 
\epsilon \frac{\logn x}{\alpha} - \log \left ( \omega_W(\logn x)\right ) - \log(1+\epsilon_1),
\end{align*}
where $\alpha' := \alpha/(1+\epsilon)$.
Since $\liminf_{x\rightarrow\infty}H(x)=\infty$,
there is $t_2$ such that $H(x) > 2$ for all $x > t_2$.
By Lemma~\ref{Lem:Gx}, we can choose $t_3$ such that
\begin{align}
\log G((1+\epsilon_1) u_n(x)) \le - \frac{1+2\epsilon}{\alpha}
x \logn x,
\label{Eq:Gup}
\end{align}
for all $x > t_3$.
From now on, we only consider large $m$ such that $m > t_0:=\max\{t_1,t_2,t_3\}$.
For convenience, we define $\tau_t := t+m$.
Let
$
E_t' :=\left  \{ Q_t \le (1+\epsilon_1)u_n(\tau_t) \right \}.
$
Since $E_0 = E_0'$ and $E_{t+1} \cap E_t' = E_{t+1}' \cap E_t'$ even though
$E_t'$ can be a proper subset of $E_t$,
we have
\begin{align}
\bigcap_{k=0}^t E_k
=\bigcap_{k=0}^t E_k'.
\label{Eq:freeE}
\end{align}
We have, for $k \ge 1$, that
\begin{align*}
\P\left (A_k\vert A_{k-1} \cap E_{k-1}'\right )
\ge 1 - 2 \exp\left (-e^{\tau_k}
\right )=:1- \xi_k,
\end{align*}
where we have used \eqref{Eq:Pup} with
$
f \mapsto (1+\epsilon_1)u_n(\tau_{k-1})\ge 1,$
$x \mapsto \exp  ( \frac{1+\epsilon}{\alpha}\tau_{k-1}\logn{\tau_{k-1}} ) \ge e^{\tau_k},$
and 
$B \mapsto e^{H(\tau_{k-1})} \ge e^2$ and the fact $S_t \le Q_t$.\medskip

Observe that
$\P(D_t) = \P(E_t \vert A_t \cap D_{t-1} ) \P(A_t \vert D_{t-1} )\P(D_{t-1})$.
Using Lemma~\ref{Lem:Markov} and \eqref{Eq:freeE}, we have
$$
\P\left (A_k \vert D_{k-1}\right )
\ge  1-\xi_k
$$
Since $W_k$ is purely determined by $\Xi(k)$, $E_k$ is independent
of $D_{k-1}$ and, accordingly, we have
$$
\P(E_k \vert  A_k \cap D_{k-1}) = \P(E_k \vert  A_k)
\ge 1 - \beta \exp \left ( -\frac{\epsilon}{\alpha} \tau_k \logn {\tau_k} \right )
 =: 1-\eta_k,
$$
where we have used \eqref{Eq:Wup} with 
$\alpha \log y \mapsto (1+\epsilon)\tau_{k} \logn{\tau_k}$,
$x \mapsto (1+\epsilon_1) u_n(\tau_k)$,
and \eqref{Eq:Gup}.
Therefore, 
\begin{align}
\P(D) \ge \prod_{k=1}^\infty (1-\xi_k)(1-\eta_k)
\ge 1 - \sum_{k=1}^\infty (\xi_k + \eta_k).
\label{Eq:Dser}
\end{align}
Note that
$\lim_{m\rightarrow\infty} (\xi_{k}+\eta_{k})=0$.
Since
$
\lim_{k\rightarrow\infty} (\xi_k+\eta_k) \tau_k^2 = 0,
$
there is a constant $c$ that is independent of $m$ such that 
$\xi_k + \eta_k \le c \tau_k^{-2} \le c k^{-2}$ for all $k$.
Hence, the series in~\eqref{Eq:Dser} converges uniformly 
for all $m>t_0$ and, therefore,
$
\lim_{m\rightarrow\infty} \P(D) = 1,$
which completes the proof.
\end{proof}

\begin{lemma}[\textbf{Upper bound}]
\label{Cor:Xup}
If $X(0)<\infty$ and $Q_0<\infty$, then
almost surely,
\begin{align*}
\limsup_{t\rightarrow\infty} \frac{\log X(t)}{t \logn t} \le \frac{1}{\alpha},\quad
\limsup_{t\rightarrow\infty} \frac{W_t}{u_n(t)} \le 1.
\end{align*}
\end{lemma}
\begin{proof}
Choose $\epsilon$ and $\epsilon_1$ as in Lemma~\ref{Lem:Gx}.
Let 
\begin{align*}
C(\epsilon) &:= \left \{ \limsup_{t\rightarrow\infty} \frac{\log X(t)}{t \logn t} \le \frac{1+\epsilon}{\alpha} \right \},\\
\tilde C(m,\epsilon) &:= \bigg \{\log X(t) \le (1+\epsilon)\frac{(t+m) \logn{t+m}}{\alpha} 
\text{ for all } t \bigg \}.
\end{align*}
We use $D$ in Lemma~\ref{Lem:Xup} with $m$ to be the same meaning as in this lemma.
Since $D \subset \tilde C(m,\epsilon) \subset C(\epsilon)$ for any $m >0$,
Lemma~\ref{Lem:Xup} gives $\P(C(\epsilon))=1$.
Defining $E = \bigcap_{t=1}^\infty E_t$ we get
\smash{$
\lim\limits_{m\rightarrow\infty} \P(E) = 1,
$}
because $D \subset E$.
Therefore,
$$
\P\left (\limsup_{t\rightarrow\infty} \frac{W_t}{u_n(t)} \le 1+\epsilon_1\right ) 
\ge \lim_{m\rightarrow\infty}\P(E)
= 1.
$$
Since $\epsilon$ is arbitrary, the proof is completed.
\end{proof}

\noindent
{\bf Definition (Initial condition for Lemma~\ref{Lem:LB}).}
Choose $\epsilon$ as in Lemma~\ref{Lem:Gx}. Let
\begin{align}
\nonumber
\alpha_1:=\alpha\left (1-\frac{\epsilon}2\right)^{-1},\quad
f_k &:= (1-\beta) u_n(k),\quad
b_k := \frac{1}{1-\beta}\exp \left ( -\frac{\epsilon}{2\alpha} \logn{k} \right ), 
\\ b_k f_k &= \exp \left ( \frac{\logn{k}}{\alpha_1} + 
\log \left ( \omega_W\left ( \logn k \right ) \right )\right ),
\label{Eq:bkfk1}
\end{align}
where $k$ is assumed sufficiently large in order for the definition 
to make sense. We also define
\begin{align}
\nonumber
H(m,x)&:= \frac{\logn m}{\alpha_1}(x-m) 
+ (x-m) \log \left ( \omega_W(\logn m) \right ),\\ 
\nonumber
h(m,x)&:= H(m,x)  - \frac{1-\epsilon}{\alpha}x \logn x,\\
\tau_j(m)&:=\en{n}{ \left (1+j \epsilon_2 \right ) \logn m},
\quad \epsilon_2 := \frac{\epsilon}{8 (1-\epsilon)}<\frac{\epsilon}{4},
\label{Eq:Hh12}
\end{align}
where $j=1,2$ and we assume $\logn m>1$, $\omega_W(\logn m) > 0$, and $x>m$. 
Note that 
$(1-\epsilon/2)/(1+2\epsilon_2) > 1 - \epsilon$.
Since $$
(b_m f_m)^{x-m}\exp\left (-\frac{1-\epsilon}{\alpha} x \logn{x}\right ) = e^{h(m,x)},
$$
$h(m,x) \ge 0$ implies 
$(b_m f_m)^{x-m}\ge \exp\left (\frac{1-\epsilon}{\alpha} x \logn{x}\right )$.
\medskip

We choose an integer $k_0$ as in Lemma~\ref{Lem:k0I}.
Once $k_0$ is fixed, we define an initial condition for any integer $t_0 \ge k_0$. In generation $0$, there are $t_0 - k_0 + 1$ different mutant types
with fitness $F_k = f_k/(1-\beta)$ ($k_0 \le k \le t_0$) and
the number $M_k(0)$ of individuals with fitness $F_k$ is 
$M_k(0)=\lceil f_k^{t_0-k} \rceil \ge (b_k f_k)^{t_0-k}$.
We denote the number of nonmutated descendants of $M_k(0)$ 
in generation $t$ by $M_k(t)$.
\medskip

For convenience, we denote the largest fitness among mutants at generation 
$k\ge 1$ by $F_{k+t_0}$ and its nonmutated descendants at generation $t\ge k$ by
$M_{k+t_0}(t)$. Note that $W_k = F_{k+t_0}$ and $N_k(t) = M_{k+k_0}(t)$ for $k\ge 1$. 
We set $F_{k+t_0}=0$ if there are no new mutants at generation $k$.
If $F_{k+t_0} = 0$, we write $M_{k+t_0}(t) = 0$ for all $t$.
If $F_{k+t_0} > 0$, we set $M_{k+t_0}(k) = 1$. 
That is, even if there are many mutants with the same largest fitness
$F_{k+t_0}$, which may frequently happen in the \mm
if discrete fitness values are allowed to be sampled, $M_{k+t_0}(t)$ only
concerns descendants of a single individual among them. 
Finally, we define
$$
\fX(t) := \sum_{k=k_0}^{t+t_0} M_k(t).$$
Note that $\fX(t) \le X(t)$ and equality holds for the \fm.

\begin{lemma}
\label{Lem:k0I}
For $b_k$, $f_k$ in \eqref{Eq:bkfk1} and for $H$, $h$, $\tau_1$, $\tau_2$ in
\eqref{Eq:Hh12},
there is an integer $k_0$, which is larger than $\en{n}1$, such that 
for all $m \ge k_0$
\begin{description}
\item[(Condition 1)] $0<b_{m} < b_c$, $(1-b_m + b_m \log b_m) f_m > 1$, and $b_m f_m > e$ (see Remark~\ref{Rem:bc} for the motivation of this condition);
\item[(Condition 2)] $h(m,x) > 0$ with any $x$ satisfying $\tau_1(m) \le x \le \tau_2(m)$.
\end{description}
\end{lemma}
\begin{proof}
It is obvious that there is an integer $k_1$ such that (\textbf{Condition 1})
is satisfied for all $m \ge k_1$.
For any sequence $(x_m)_{m \ge \lceil \en{n}1 \rceil}$ such that
$\tau_1(m) \le x_m \le \tau_2(m)$, we have
\begin{align}
\liminf_{m\rightarrow\infty} \frac{\alpha H(m,x_m)}{x_m \logn{x_m}}
\ge
\frac{1-\epsilon/2}{1 + 2\epsilon_2} \liminf_{m\rightarrow\infty} \frac{(x_m-m) \logn{m}}{ x_m \logn{m}}
=\frac{1-\epsilon/2}{1 + 2\epsilon_2} 
> 1-\epsilon,
\label{Eq:hmxt}
\end{align}
where we have used $\logn{x_m} \le (1+2\epsilon_2) \logn{m}$ by assumption,
$x_m \ge \tau_1(m) \ge m^{1+\epsilon_2}$ for $m \ge \en{n}1$ (Lemma~\ref{Lem:expnd}),
and $\omega_W$ is a slowly varying function.
Therefore, there is an integer $k_2$ such that $h(m,x) > 0$ for all $m \ge k_2$
with any $x$ satisfying $\tau_1(m)\le x\le \tau_2(m)$.
Now we set $k_0 = \max\{\lceil \en{n}{1}\rceil, k_1,k_2\}$ and the
proof is completed.
\end{proof}
\begin{remark}
\label{Rem:hmx}
Inequality \eqref{Eq:hmxt} is valid even if we relax the lower bound of $x_m$ as long as
$\lim_{m\rightarrow\infty} m/x_m = 0$. For example, replacing $\tau_1(m)$ by $m \sqrt{\log m}$
still gives $h(m,x) >0$ for sufficiently large $m$.
\end{remark}
\begin{lemma}
\label{Lem:LB}
We fix $\epsilon$ and $\epsilon_1$ as in Lemma~\ref{Lem:Gx}.
We also use the initial conditions defined above with $t_0 \ge k_0$ and define
\begin{align*}
E_t&:=\left \{\log \fX(t) \ge \frac{1-\epsilon}{\alpha} (t+t_0)\logn{t+t_0} 
\right \}, \quad E := \bigcap_{t=1}^\infty E_t,\\
J_t&:=\left \{W_{t} \ge (1-\epsilon_1)u_n(t+t_0)\right \}, \quad J:= \bigcap_{t=1}^\infty J_t.
\end{align*}
Then,
$$
\lim_{t_0\rightarrow\infty}
\P\left ( E \right )=
\lim_{t_0\rightarrow\infty}
\P\left ( J \right )=1.
$$
\end{lemma}
\begin{proof}
We define a sequence $(m_\ell)_{\ell\ge 0}$ as
\begin{align*}
m_\ell := \left \lfloor \en{n-1}{\ell +\lo{n-1}{\tfrac{t_0}{\log t_0}} } \right \rfloor.
\end{align*}
We first work out how large $t_0$ should be.
Obviously, there exists $t_1$ such that
$k_0<m_0 < t_0$ for all $t_0 > t_1$.
Since
\begin{align*}
\lim_{t_0\rightarrow\infty} \frac{t_0}{\tau_2(m_0)} = 0,\quad
\lim_{t_0\rightarrow\infty} \frac{\alpha H(m_0,t_0+1)}{(t_0+1) \logn{t_0+1}} = 1> 1-\epsilon,
\end{align*}
there is $t_2$ such that
$\tau_2(m_0) \ge t_0+1$ and $h(m_0,t_0+1)>0$ for all $t_0 > t_2$.
In fact, $h(m_0,x) > 0$ for all $t_0+1 \le x \le \tau_2(m_0)$; see Remark~\ref{Rem:hmx}.
Since
$$
\lim_{t_0\rightarrow\infty} \frac{\logn{m_\ell}}{\logn{m_{\ell-1}}}=
\lim_{\ell\rightarrow\infty} \frac{\logn{m_\ell}}{\logn{m_{\ell-1}}}=1,
$$
there is $t_3$ such that
$\tau_2(m_{\ell-1}) \ge \tau_1(m_{\ell})$ for all $t_0 > t_3$ and
for all $\ell \ge 1$.
By Lemma~\ref{Lem:Gx}, there is $t_4$ such that
\begin{align}
\log G\left ( (1-\epsilon_1)u_n(t+t_0) \right ) \ge -\frac{1-2\epsilon}{\alpha} (t+t_0) \logn {t+t_0},
\label{Eq:Ge0}
\end{align}
for all $t_0 > t_4$ and for all $t \ge 0$.
For later references, we define sequences $(\xi_\ell)_{\ell \ge 0}$
and $(\eta_t)_{t \ge 0}$ as
\begin{align*}
\eta_t &:= \exp \left (  - \beta \exp \left ( \frac{\epsilon}{\alpha}
(t+t_0) \logn{t+t_0} \right ) \right ),\\
\xi_\ell &:= 2 (1-\beta)u_n(m_\ell) \exp\left (-\tfrac{1-\beta}2 u_n(m_\ell)\right ).
\end{align*}
Since $\logn{x}$ is unbounded and increasing function for large $x$,
there is $t_5$ such that 
$\epsilon \logn{t_0} \ge \alpha$
for all $t_0>t_5$ and all $\ell \ge 0$. 
Therefore, we have 
$\eta_t \le \exp\left ( - \beta e^t \right ),$
for all $t_0 > t_5$.
This implies that
$\sum_t \eta_t$ is uniformly convergent
for all $t_0$ and therefore,
\begin{align}
\lim_{t_0\rightarrow\infty} \sum_{t=1}^\infty \eta_t 
=\sum_{t=1}^\infty \lim_{t_0\rightarrow\infty} \eta_t 
= 0.
\label{Eq:etasum}
\end{align}
Since $4 x e^{-x} \le e^{-x/2}$ for $x \ge 10$,
we have 
$
\xi_\ell \le \exp ( -\frac{1-\beta}{4} u_n(m_\ell) ),
$
for $(1-\beta)u_n (m_\ell) \ge 10$.
Since 
$$
\lim_{x\rightarrow\infty} \frac{4 \left ( \lo{n-1}x\right )^{1/\alpha_1}}{(1-\beta) u_n(x)} = 0,
$$ 
there is $t_6$ such that
$(1-\beta)u_n(m_\ell) \ge 10$ for any $\ell \ge 0$ and
$$
\frac{1-\beta}{4} u_n(m_\ell) \ge \left ( \lo{n-1}{m_\ell} \right )^{1/\alpha_1}
\ge \ell^{1/\alpha_1},
$$
for all $t_0 > t_6$.
Note that, under this assumption, we have $\xi_\ell \le \exp\left (-\ell^{1/\alpha_1}\right )$,
which shows that
$\sum_{\ell=0}^\infty \xi_\ell $ converges uniformly for all
large $t_0$ and therefore,
\begin{align}
\lim_{t_0\rightarrow\infty} \sum_{\ell=0}^\infty \xi_\ell =
\sum_{\ell=0}^\infty \lim_{t_0\rightarrow\infty} \xi_\ell =0.
\label{Eq:xisum}
\end{align}
In the following, we assume $t_0 > \max\{t_1, t_2,t_3,t_4,t_5,t_6\}$.\medskip

Now we are ready for the proof.
We first define two sequences $(a_\ell)_{\ell\ge 0}$ and $(u_\ell)_{\ell\ge 0}$ such that
$a_0 = 0$,  $a_\ell = \tau_1(m_\ell)-t_0$ for $\ell \ge 1$, and
$u_\ell = \tau_2(m_\ell)-t_0$ for $\ell \ge 0$.
Note that $a_{\ell+1} \le u_\ell$ for all $\ell \ge 0$ and $m_\ell < t_0+a_\ell$.
Notice also that for $a_\ell \le t < a_{\ell+1} \le u_\ell$ 
\begin{align}
(t+t_0 - m_\ell)\log (b_{m_\ell} f_{m_\ell}) \ge  \frac{1-\epsilon}\alpha
(t+t_0) \logn{t+t_0},
\label{Eq:Econd}
\end{align}
which also implies $E_0$ is an almost sure event.
For $t\ge 0$, we define
$$
A_t := \left \{ M_{m_{\ell'}} \ge \left ( b_{m_{\ell'}} f_{m_{\ell'}} \right )^{t+t_0-m_{\ell'}}
\right \},
$$
where $\ell'$ is (uniquely) determined by the condition $a_{\ell'} \le t < a_{\ell'+1}$.
Note that $A_t \subset E_t$.
Define 
$$
\tilde J := \bigcap_{m=m_0}^{t_0} \{ F_m \ge f_m/(1-\beta)\},\quad
C_0 := A_0 \cap \tilde J, \quad C_t := J_t \cap A_t \cap C_{t-1},\quad
C := \bigcap_{t=1}^\infty C_t.$$
Note that $\tilde J$ and $A_0$ are sure events and so is $C_0$.
Observe that
$$\P(C_t) = \P(J_t \vert A_t \cap C_{t-1} ) \P(A_t \vert C_{t-1}) \P(C_{t-1}).$$
Since $W_t$ is solely determined by $\Xi(t)$ and
$\alpha \log \Xi(t) \ge (1-\epsilon)(t+t_0) \logn{t+t_0}$ in the event
$A_t\cap C_{t-1}$, 
we have
$
\P(J_t \vert A_t \cap C_{t-1}) \ge 1 - \eta_t,
$
where we have used \eqref{Eq:Wdown} with 
$\alpha \log y \mapsto (1-\epsilon) (t+t_0)\logn{t+t_0}$,
$x \mapsto (1-\epsilon_1) u_n(t)$ with \eqref{Eq:Ge0}.
Therefore, we have 
$$
\P(C_t) \ge \left (\prod_{\tau=1}^t \left (1-\eta_\tau\right )
\right ) 
\prod_{\ell=0}^{\ell'}
P_\ell,\quad P_\ell:=\prod_{\tau=a_\ell}^{a_{\ell+1}-1}
\P(A_\tau \vert C_{\tau-1} ),
$$
where we have used the fact that probability cannot be larger than 1.\medskip

Let us find the lower bound of $P_\ell$. Assume $a_\ell \le \tau < a_{\ell+1}$.
Note that $A_\tau$ is independent of $J_k$ for $a_\ell \le k < \tau$ (this 
is because $m_\ell < a_\ell + t_0$)
and of $A_k$ for $k < a_\ell$ (this is because $M_m(t)$'s for different $m$'s are mutually independent branching
processes). Therefore,
$$
\P(A_\tau \vert C_{\tau-1}) = 
\P\bigg (A_\tau \Big \vert \Big (\bigcap_{k=a_\ell}^{\tau-1} A_k \Big )
\Big .\cap J_{m_\ell-t_0}\bigg ),
$$
where $J_{m_\ell-t_0}$ for $m_\ell< t_0$ should be
interpreted as $\tilde J$.
By simple algebra, we get
\begin{align*}
P_\ell &= \prod_{\tau=a_\ell}^{a_{\ell+1} -1}
\P\bigg (A_\tau \Big\vert \Big(\bigcap_{k=a_\ell}^{\tau-1} A_k \Big)
 \cap J_{m_\ell-t_0}\bigg)
= \P\bigg (  \bigcap_{\tau=a_\ell}^{a_{\ell+1}-1} A_\tau \Big\vert J_{m_\ell - t_0} \bigg )\\
&= \P\left ( M_{m_\ell} \ge (b_{m_\ell} f_{m_\ell} )^{k+t_0-m_\ell}
\text{ for all } a_\ell \le k < a_{\ell+1}-1 \vert F_{m_\ell} \ge f_{m_\ell}/(1-\beta) \right )\\
&\ge \P\left ( M_{m_\ell} \ge (b_{m_\ell} f_{m_\ell} )^{k+t_0-m_\ell}
\text{ for all }  k \ge 0 \vert F_{m_\ell} \ge f_{m_\ell}/(1-\beta)\right )
\ge 1 - \xi_\ell,
\end{align*}
where we have used \eqref{Eq:Plow} with $f \mapsto f_{m_\ell}$.
Therefore,
$$
\P(C_t) \ge \left ( \prod_{\tau=1}^t (1 - \eta_\tau) \right )
\left ( \prod_{\ell=0}^{\ell'} ( 1 - \xi_\ell ) \right )
\ge 1 - \sum_{\tau=1}^t \eta_\tau - \sum_{\ell=0}^{\ell'} \xi_\ell.
$$
By \eqref{Eq:etasum} and \eqref{Eq:xisum}, we have
$$
\lim_{t_0\rightarrow\infty} \P(C) = 1.
$$
Since $C \subset E$ and $C \subset J$, the proof is completed.
\end{proof}
\begin{lemma}[\textbf{Lower bound}]
\label{Cor:Xlow}
Almost surely on survival,
\begin{align*}
\liminf_{t\rightarrow\infty} \frac{\log X(t)}{t \logn t} \ge \frac{1}{\alpha},\quad
\liminf_{t\rightarrow\infty} \frac{W_t}{u_n(t)} \ge 1.
\end{align*}
In other words,
\begin{align*}
\P\left (\liminf_{t\rightarrow\infty} \frac{\log X(t)}{t \logn t} \ge \frac{1}{\alpha}\right )=
\P\left ( 
\liminf_{t\rightarrow\infty} \frac{W_t}{u_n(t)} \ge 1
\right ) 
= \P(\Su)
=p_s.
\end{align*}
\end{lemma}
\begin{proof}
Fix $\epsilon$ and $\epsilon_1$ as in Lemma~\ref{Lem:Gx}.
For any $0<\epsilon'$, Lemma~\ref{Lem:LB} implies the existence of $t_0$ such that
\begin{align*}
\P\left (
\log \fX(t)
\ge
\frac{1-\epsilon}{\alpha }(t+t_0)\logn{t+t_0}
\text{ for all } t\ge 0 \right ) \ge 1- \epsilon',\\
\P\left (
W_t
\ge
(1-\epsilon_1) u_n(t+t_0)
\text{ for all } t\ge 0 \right )\ge 1- \epsilon'.
\end{align*}
Since $W_t$ as well as $X(t)$ is unbounded on survival (Lemma~\ref{Lem:Wt}),
there should be $\tau$ and $k\ge 1$ almost surely on survival such that $W_\tau > 
(1-\epsilon_1)u_n(t_0)$
and $N > \fX(0)$, where $N$ is the number of individual with
fitness $W_\tau$ at generation $\tau+k$.
Now couple $X(t+\tau+k)$ 
with $\fX(t)$, which gives $X(t+\tau+k) \ge \fX(t)$ for all $t\ge 0$.
We denote the event that has such $\tau$ and $k$ by $D$. Note that 
$\P(D \cap \Su)=p_s$ by Lemma~\ref{Lem:Wt} and, obviously, $\P(D) \ge p_s$.
Therefore,
\begin{align*}
p_s&\ge
\P\left ( 
\liminf_{t\rightarrow\infty} \frac{\log X(t)}{t \logn t} \ge \frac{1-\epsilon}{\alpha}
\right ) \\
&\ge
\P\left ( 
\left . \liminf_{t\rightarrow\infty} \frac{\log X(t)}{t \logn t} \ge \frac{1-\epsilon}{\alpha}
\right \vert D\right ) \P(D)\\
&\ge
\P\left (
\log \fX(t)
\ge
\frac{1-\epsilon}{\alpha}(t+t_0)\logn{t+t_0}
\text{ for all } t\ge 0 \right )\P(D)
\ge (1-\epsilon') p_s,
\end{align*}
where we have used the Markov property.
By the same token, we have
\begin{align*}
p_s\ge
\P\left ( \liminf_{t\rightarrow\infty} \frac{W_t}{u_n(t)} \ge 1-\epsilon_1
\right ) \ge (1-\epsilon')p_s.
\end{align*}
Since $\epsilon'$ and $\epsilon$ are arbitrary, the proof is completed.
\end{proof}
By Lemma~\ref{Cor:Xup} and Lemma~\ref{Cor:Xlow}, Theorem~\ref{Th:mainthm} is proved.

\section{\label{Sec:typeii} Proof of Theorem~\ref{Th:mainII}}
This section presents two lemmas, which will prove Theorem~\ref{Th:mainII}.
Needless to say, $G$ is always of type II throughout this section.
For convenience, we define
\begin{align*}
\chi(t,n,\nu) &:= \begin{cases}
t^\nu, &n=1,\\
\exp (t^\nu),& n = 2,\\
\exp\Big( t\big(  \lo{n-2}t \big)^{-\nu} \Big), & n\ge 3,
\end{cases}\\ 
\uii(t,n,\nu,a) &:= \chi(t,n,\nu) \left ( \lo{\max\{0,n-2\}}{t} \right )^{-a},
\quad \tG(x,n,a):= \log x \left ( \lo n x \right )^a,
\end{align*}
with an appropriate domain.
Again, the distinction between the \mm and the \fm does not play any role in the
proof of Theorem~\ref{Th:mainII}.
\begin{lemma}[\textbf{Variation of Lemma~\ref{Lem:Xup}}]
\label{Lem:Xuptype22}
Assume $X(0) < \infty$ and $Q_0< \infty$, fix $\epsilon>0$ and~let 
\begin{align*}
\nu_n &:= 
\begin{cases}
(1+2 \epsilon)(1+\alpha)/\alpha,&n=1,\\
\epsilon+1/(1+\alpha ), & n=2,\\
\alpha/(1+\epsilon)^2, & n \ge 3,
\end{cases}\quad
a_n:=
\begin{cases}
1+\epsilon,&n=1,\\
\alpha/(1+\alpha),&n=2,\\
\alpha/(1+\epsilon),&n\ge 3.
\end{cases}
\end{align*}
Then
\begin{align*}
\lim_{m\rightarrow\infty} \P\left ( \log \Xi(t)  \le 
\chi(t+m,n,\nu_n),\,
\log W_t \le \uii(t+m,n,\nu_n,a_n)
\text{ for all } t\right ) 
&= 1.
\end{align*}
\end{lemma}
\begin{proof}
We first make a precise criterion as to the meaning of large $m$.
Obviously,
there is $m_1$ such that $\chi(m,n,\nu_n) \ge \log X(0)$
and $\uii(m,n,\nu_n,a_n) \ge \log Q_0$ for all $m > m_1$.
Let 
$
H(x):= 
\chi(x+1,n,\nu_n)
-\chi(x,n,\nu_n)
- \uii(x,n,\nu_n,a_n)$.
By the mean value theorem, there is $x_0$ ($x \le x_0 \le x+1$)
such that
\begin{align*}
&\chi(x+1,n,\nu_n) - \chi(x,n,\nu_n)
= \left . \frac{\partial \chi(x,n,\nu_n)}{\partial x} 
\right \vert_{x=x_0}\\
&=\uii(x_0,n,\nu_n,a_n) 
\times 
\begin{cases}
\nu_n x_0^{\epsilon},&n\le 2,\\
\displaystyle
\left ( \lo{n-2}{x_0} \right )^{\epsilon \nu_n}
\left ( 1- \nu_n \prod_{k=1}^{n-2} \frac{1}{\lo k {x_0}} \right ), & n\ge 3,\end{cases}
\end{align*}
which gives
$\lim_{x\rightarrow\infty} H(x) = \infty$.
Therefore, there is $m_2$ such that $H(x) > 2$ for all $x > m_2$.
Let 
$\epsilon_0 = 
\epsilon/(1+\epsilon).
$
By definition, there is $m_3$ such that 
$\log G(x) \le - \tG(x,n,\alpha/(1+\epsilon_0))$ for all $x>m_3$.
Since $(\nu_1-a_1)\alpha > a_1 (1+\epsilon_0)$,
$\nu_2 \alpha > a_2 (1+\epsilon_0)$, and $\alpha > a_n(1+\epsilon_0)$ for $n \ge 3$,
we have
\begin{align*}
&\lim_{t\rightarrow\infty}\frac{\tG\left ( \exp\left ( \uii(t,n,\nu_n,a_n) \right ),n,\alpha/(1+\epsilon_0)\right )}{\chi(t,n,\nu_n)}\\
&=\lim_{t\rightarrow\infty} \left ( \lo{\max\{0,n-2\}}{t} \right )^{-a_n} \left (\lo{n-1}{\uii(t,n,\nu_n,a_n)}\right )^{\alpha/(1+\epsilon_0)}=\infty.
\end{align*}
Therefore, there is $m_4$ such that
$\tG\left ( \exp\left ( \uii(t,n,\nu_n,a_n) \right ),n,\alpha/(1+\epsilon_0)\right ) 
\ge 2 \chi(t,n,\nu_n)$ and, accordingly,
\begin{align}
G\left ( \exp \left ( \uii(t,n,\nu_n,a_n)\right ) \right ) &\le 
e^{- 2 \chi(t,n,\nu_n)},
\label{Eq:Gup2}
\end{align}
for all $t > \max\{m_3,m_4\}$.
We set $m_0 = \max\{m_1,m_2,m_3,m_4\}$ and
we assume $m > m_0$ in what follows.
For given $m$, we define $\tau_t := t+m$ and
\begin{align*}
E_t &:=\left  \{ \log W_t \le \uii(\tau_t,n,\nu_n,a_n) \right \},\quad
E_t':=\left  \{ \log Q_t \le \uii(\tau_t,n,\nu_n,a_n) \right \},\\
A_t &:=\left \{\log \Xi(t) \le \chi(\tau_t,n,\nu_n )\right \},\quad
A=\bigcap_{k=1}^\infty A_k.
\end{align*}
We can repeat \eqref{Eq:freeE} for $E_t$ and $E_t'$.
\begin{align}
\bigcap_{k=0}^t E_k
=\bigcap_{k=0}^t E_k'.
\label{Eq:freeE2}
\end{align}
We also define, for $t\ge 1$,
$
D_0 = A_0 \cap E_0,$ 
$D_t = A_t \cap E_t \cap D_{t-1},$ and
$D = \bigcap_{k=1}^\infty D_k.
$
Observe that
$\P(D_t) = \P(E_t \vert A_t \cap D_{t-1} ) \P(A_t \vert D_{t-1} )\P(D_{t-1})$.
Using Lemma~\ref{Lem:Markov} and \eqref{Eq:freeE2}, we have
$$
\P(A_k \vert D_{k-1})
= \P(A_k \vert A_{k-1}\cap E'_{k-1})
\ge 1 - 2 \exp\left (-e^{\uii(\tau_{k-1},n,\nu_n,a_n)+\chi(\tau_{k-1},n,\nu_n)
}\right )=:1- \xi_k,
$$
where we have used \eqref{Eq:Pup} with
$f \mapsto e^{\uii(\tau_{k-1},n,\nu_n,a_n)}$,
$x \mapsto e^{\chi(\tau_{k-1},n,\nu_n)}$, and
$B \mapsto e^{H(\tau_{k-1})} \ge e^2$.
Since $W_t$ is purely determined by $\Xi(t)$, we have
\begin{align*}
\P(E_t \vert  A_t \cap D_{t-1}) = \P(E_t \vert  A_t)
&\ge 1-\beta\exp \left ( - 
\chi(\tau_k,n,\nu_n,a_n)\right )
=: 1-\eta_k,
\end{align*}
where we have used \eqref{Eq:Wup} with
$y \mapsto e^{\chi(\tau_k,n,\nu_n)}$, $x \mapsto e^{\uii(\tau_k,n,\nu_n,a_n)}$,
and \eqref{Eq:Gup2}.
Therefore, we have
\begin{align}
\P(D) \ge \prod_{k=1}^\infty (1-\xi_k)(1-\eta_k)
\ge 1 - \sum_{k=1}^\infty (\xi_k + \eta_k).
\label{Eq:Dser2}
\end{align}
Since $\lim_{k\rightarrow\infty} (\xi_k + \eta_k) \tau_k^2 =0$ and
$\tau_k^{-2} < k^{-2}$, 
the series in \eqref{Eq:Dser2} converges uniformly for large $m$.
Since $\lim_{m\rightarrow\infty} (\xi_k+\eta_k) = 0$ for all $k$, 
we have
$
\lim_{m\rightarrow\infty} \P(D) = 1$,
which completes the proof.
\end{proof}

\noindent
{\bf Definition (Initial condition for Lemma~\ref{Lem:LB22}).}
Fix $0<\epsilon<1/\alpha$ and let 
\begin{align*}
\nu_n&:= 
\begin{cases}
(1+\alpha)/[\alpha(1+2\epsilon)] , & n=1,\\
1/ [(1+\alpha)(1+2\epsilon)], & n =2,\\
\alpha (1+3\epsilon), &n\ge 3,
\end{cases}
\quad
a_n :=
\begin{cases}
(1+\alpha)/(1+\alpha + \alpha \epsilon), & n= 1,\\
\alpha/(1+\alpha), & n= 2,\\
\alpha  (1 + 2\epsilon), & n\ge 3,
\end{cases}\quad
\end{align*}
which should not be confused with $\nu_n$ and $a_n$ defined in Lemma~\ref{Lem:Xuptype22}.
Note that $\nu_1 > a_1$ because $\epsilon<1/\alpha$.
Define
\begin{align*}
f_k&:=(1-\beta) \exp\left ( \uii(k,n,\nu_n,a_n)\right ),\quad
b_k:=\frac1{1-\beta} \exp\left ( - \frac{\epsilon}{1+\epsilon}\uii(k,n,\nu_n,a_n) \right ),\\
f_k b_k &= \exp\left ( \frac{\uii(k,n,\nu_n,a_n)}{1+\epsilon} \right ).
\end{align*}
Once $k_0$ is determined as in Lemma~\ref{Lem:k0II}, we define the initial condition with an integer $t_0$ larger than $k_0$ in exactly the same way as in the previous section.
We use $M_k(t)$, $F_k$, $F_{k+t_0}$, and $\fX(t)$ with an appropriate modification
of the meaning.

\begin{lemma}
\label{Lem:k0II}
For $\epsilon$, $\nu_n$, $a_n$, $b_k$, $f_k$ defined above
there is an integer $k_0$, which is larger than $\en{n}0$, such that 
for all $m \ge k_0$ we have
\begin{description}
\item[(Condition 1)] $0<b_{m} < b_c$, $(1-b_m + b_m \log b_m) f_m > 1$, and $b_m f_m > e$;
\item[(Condition 2)] 
$G\left ( \exp(\uii(m,n,\nu_n,a_n)) \right ) \ge \exp \left ( -\frac{1}{2} \chi(m,n,\nu_n)\right )$.
\end{description}
\end{lemma}
\begin{proof}
Obviously, there is an integer $k_1$ that satisfies (\textbf{Condition 1}) for all
$m \ge k_0$.
By definition, we have $\log G(x) \ge -\tG(x,n,\alpha(1+\epsilon))$
for all sufficiently large $x$. 
Since $(\nu_1 - a_1) \alpha (1+\epsilon) < a_1$, $\nu_2 \alpha (1+\epsilon)<a_2$,
and $a_n > \alpha(1+\epsilon)$ for $n\ge 3$,
we have 
\begin{align*}
&\lim_{y\rightarrow\infty}\frac{\tG(\exp(\uii(y,n,\nu_n,a_n)))}{\chi(y,n,\nu_n)}\\
&=\lim_{y\rightarrow\infty} \left ( \lo{\max\{0,n-2\}}{y} \right )^{-a_n} \left (\lo{n-1}{\uii(y,n,\nu_n,a_n)}\right )^{\alpha(1+\epsilon)}=0,
\end{align*}
which guarantees the existence of an integer $k_2$ such that
\begin{align}
\log G\left ( \exp(\uii(y,n,\nu_n,a_n)) \right ) \ge -\frac{1}{2} \chi(y,n,\nu_n)
\label{Eq:Ge02}
\end{align}
for all $y \ge k_2$. Now we set $k_0 := \max\{k_1,k_2\}$, which completes the proof.
\end{proof}
\begin{lemma}[\textbf{Variation of Lemma~\ref{Lem:LB}}]
\label{Lem:LB22}
For the initial conditions defined above
with $t_0 \ge k_0$, we define two events
\begin{align*}
E_t&:=\{\log \fX(t) \ge \chi(t+t_0,n,\nu_n)\},\quad
E:=\bigcap_{t=1}^\infty E_t,\\
J_t&:=\{\log W_{t} \ge \uii(t+t_0,n,\nu_n,a_n)\},\quad
J:=\bigcap_{t=1}^\infty J_t.
\end{align*}
Then,
$$
\lim_{t_0\rightarrow\infty}
\P\left ( E \right )=
\lim_{t_0\rightarrow\infty}
\P\left ( J \right )=1.
$$
\end{lemma}
\begin{proof}
Let 
$$
m_t := 
\begin{cases}
\left \lfloor \frac{1}{2} ( t + t_0 )\right \rfloor, & n = 1,\\
\big \lfloor t + t_0 - \frac{1}{2}(t+t_0)^{1-\nu_2} \big \rfloor,&n=2,\\
\big \lfloor t + t_0 -\frac{1}{2} \big( \lo{n-2}{t+t_0} \big)^{\nu_n} \big \rfloor,&n\ge 3.
\end{cases}
$$
Assume $t_0$ is so  large that $m_0 > k_0$ and $(m_t)_{t\ge 0}$ is an non-dereasing sequence of $t$.
Since $1>a_1$, $1> \nu_2 + a_2$, and $\nu_n > a_n$ for $n\ge 3$,
we have
$$
\lim_{t_0\rightarrow\infty} \frac{\chi(t+t_0,n,\nu_n)}{(t+t_0-m_t) \uii(m_t,n,\nu_n,a_n)} = 
\lim_{t\rightarrow\infty} \frac{\chi(t+t_0,n,\nu_n)}{(t+t_0-m_t) \uii(m_t,n,\nu_n,a_n)} = 0.
$$
So there is $t_1$ such that
$(t+t_0-m_t) \uii(m_t,n,\nu_n,a_n) \ge (1+\epsilon)\chi(t+t_0,n,\nu_n)$
for all $t_0 > t_1$ and $t\ge 0$. In the following, we  assume $t_0 > t_1$. Define 
\begin{align*}
&A_t := \left \{ M_{m_t}(t)\ge (b_{m_t} f_{m_t})^{t+t_0-m_t}\right \},
\quad \tilde J := \bigcap_{k=m_0}^{t_0}\left \{ F_{k} \ge f_{k}/(1-\beta)\right \},\\
&C_0 = A_0 \cap \tilde J,\quad C_{t} = A_{t}\cap J_t \cap C_{k-1},\quad C = \bigcap_{t=1}^\infty C_t.
\end{align*}
Note that $\tilde J$ and $A_0$ are sure events (by the initial condition) and so is $C_0$. Also note that $A_t \subset E_t$.
Observe that
$
\P(C_t) = \P\left (J_t \vert A_t \cap C_{t-1} \right ) \P\left (A_t \vert C_{t-1} \right ) \P(C_{t-1}).
$
Since $W_t$ is solely determined by $\Xi(t)$ and
$\log \Xi(t) \ge \chi(t+t_0,n,\nu_n)$ on the event
$A_t\cap C_{t-1}$, 
we have
$$
\P(J_t \vert A_t \cap C_{t-1}) \ge 1 - \exp\left ( -\beta e^{\chi(t+t_0,n,\nu_n)/2} \right )
=: 1 - \eta_t,
$$
where we have used \eqref{Eq:Wdown} with 
$y \mapsto \exp(\chi(t+t_0,n,\nu_n))$,
$x \mapsto \exp(U(m_t,n,\nu_n,a_n))$, \eqref{Eq:Ge02}, and $\chi(m_t,n,\nu_n) \le \chi(t+t_0,n,\nu_n)$.
Therefore, we have 
$$
\P(C) \ge \left (\prod_{\tau=1}^\infty \left (1-\eta_\tau\right )
\right ) 
\prod_{\tau=1}^{\infty}
\P(A_\tau \vert C_{\tau-1} ).
$$

Note that $A_\tau$ is independent of $J_k$ for $m_\tau < k < \tau$ 
and of $A_k$ for $k < a_\ell$ (this is because $M_m(t)$'s for different $m$'s are mutually independent branching
processes). Since $m_{t+1}-m_t \le 1$, all $F_k$ with $k \ge m_0$ should affect a certain
$A_\tau$ at least once. Therefore,
\begin{align*}
\prod_{\tau=1}^\infty \P(A_\tau \vert C_{\tau-1}) &\ge 
\prod_{\tau=1}^\infty \P\left ( M_{m_\tau}(k) \ge (b_{m_\tau} f_{m_\tau} )^{k+t_0-m_\tau}
\text{ for all }  k \ge 0 \vert F_{m_\tau} \ge f_{m_\tau}/(1-\beta)\right )\\
&\ge \prod_{\tau=1}^\infty \left ( 1 - \xi_\tau \right )
\ge 1 - \sum_{\tau=1}^\infty \xi_\tau,
\end{align*}
where we have used \eqref{Eq:Plow} with $f \mapsto f_{m_\tau}$.
Therefore,
\begin{align}
\P(C) \ge 1 - \sum_{\tau=1}^t \left ( \eta_\tau + \xi_\tau \right ).
\label{Eq:pd2}
\end{align}

Since
$
\lim_{k\rightarrow\infty} (\xi_{k} + \eta_{k}) (k+t_0)^2 = 0,
$
there is a constant $c$ that is independent of $t_0$ such that 
$\xi_{k} + \eta_{k} \le c (k+t_0)^{-2} \le ck^{-2}$ for all $k$.
Therefore, the series in \eqref{Eq:pd2} converges uniformly.
Since 
$
\lim_{t_0\rightarrow\infty} (\xi_{k} + \eta_{k})  = 0
$
for any $k$, 
we have
$
\lim_{t_0\rightarrow\infty} \P(C) = 1.
$
Since $C \subset E$ and $C \subset J$, we get the desired result.
\end{proof}

By the same logic as in Lemma~\ref{Cor:Xup}
and Lemma~\ref{Cor:Xlow}, Lemma~\ref{Lem:Xuptype22} and Lemma~\ref{Lem:LB22} now prove Theorem~\ref{Th:mainII}. 

\section{\label{Sec:pfemp} The empirical fitness distribution}

In this section, 
we introduce two variants of the \fm 
that (completely or partially) neglect fluctuations in the original model
with the type I tail function.
These variants will be called the deterministic \fm (\dfm)
and semi-deterministic \fm (\sfm) and will be defined in Section~\ref{Sec:deter}
and Section~\ref{Sec:sfm}, respectively.
As we will see presently, neglecting some fluctuations will 
facilitate rigorous proofs for the limit behaviour of the \efd.\medskip

To explain the motivation of introducing the \dfm and \sfm, 
we begin by finding in Lemma~\ref{Lem:strict} tighter bounds for 
$\X_t$ of the Galton-Watson process, which 
show that the fluctuations of $N_k(t)$ become smaller and smaller over time.

\subsection{\label{Sec:emprig}Fluctuations of $N_k(t)$ and $W_t$}
\begin{lemma}
\label{Lem:bBsum}
If $B>1$, $\theta>0$, and $1\le x_1 < x_2-1$, then
$$
\P(\X_t \le Bx_2\theta \vert  x_1 \le \X_{t-1} \le x_2)\ge 1 - \frac{B}{B-1} e^{-x_1 \theta  (B \log B - B + 1)}.
$$
\end{lemma}
\begin{proof}
Let $m\ge x_1$ and $B' = B x_2 /m \ge B$. By Lemma~\ref{Lem:upsum}
together with Remark~\ref{Rem:nPoi}, we have
\begin{align*}
\P&(\X_t > Bx_2\theta \vert  \X_{t-1}=m)
=\P(\X_t > B' m\theta \vert  \X_{t-1}=m)\\
&\le \tfrac{B'}{B'-1} e^{-m \theta (B' \log B' - B' + 1)}
\le \tfrac{B}{B-1}e^{-m\theta (B \log B - B + 1)} ,
\end{align*}
where we have used the fact that $y/(y-1)$ and $y(1-\log y)$ are decreasing functions in the region
$y >1$.
Since $m \ge x_1$, we have the desired result.
\end{proof}
\begin{lemma}
\label{Cor:bBsum}
If $1<B<\frac3 2$, $0<b<1$, $(1-b+b\log b)\theta>1$, and $1\le x_1 < x_2-1$, then
$$
\P(b x_1 \theta \le \X_t \le Bx_2\theta \vert  x_1 \le \X_{t-1} \le x_2)\ge 1 
- x_1\theta e^{-x_1 \theta  (1-b)^2/2}
- \frac{B}{B-1} e^{-x_1 \theta  (B-1)^2/3}.
$$
\end{lemma}
\begin{proof}
Using Lemma~\ref{Lem:GWlow} with $f=\theta$ and Lemma~\ref{Lem:bBsum}, we have
$$
\P(b x_1 \theta \le \X_t \le Bx_2\theta \vert  x_1 \le \X_{t-1} \le x_2)\ge 1 
- x_1\theta e^{-x_1 \theta  (1-b+b \log b)}
- \frac{B}{B-1} e^{-x_1 \theta  (1-B+B \log B)}.
$$
Since
$1 - x + x \log x \ge
\begin{cases} \frac{1}{2}(1-x)^2, & 0<x<1,\\
\frac{1}{3}(x-1)^2 ,& 1<x<3/2,
\end{cases}$
the proof is  completed.
\end{proof}\smallskip

\begin{lemma}
\label{Lem:strict}
Fix $0<\epsilon<\frac{1}{2}$ 
and abbreviate $c:=(1-\epsilon)/2$. Let $a_t := \theta^{-(1-2 \epsilon)/2} \left ( 1 - \theta^{- ct} \right )$
and
\begin{align*}
b_t &:= \frac{1-a_t}{1-a_{t-1}}=1 - 
\frac{\theta^c - 1}{1-a_{t-1}} \theta^{-ct - (1-2 \epsilon)/2},
\quad \prod_{k=1}^t b_k = 1-a_t.
\\ 
B_t &:= \frac{1+a_t}{1+a_{t-1}}=1+
\frac{\theta^c - 1}{1+a_{t-1}} \theta^{-ct - (1-2\epsilon)/2},
\quad \prod_{k=1}^t B_k = 1+a_t,
\end{align*}
where $\theta$ is assumed so large that for all $t\ge 1$
\begin{align}
\label{Eq:Bbcond}
0<a_t < 1,\quad 1<B_t < \frac{3}{2},\quad \theta^t \ge \theta^{ct - (1-2\epsilon)/2}, \nonumber\\
\frac{(1-\theta^{-c})^2}{2(1-a_{t-1})} \ge \frac{1}{4},\quad
\frac{(1-a_{t-1})(1-\theta^{-c})^2}{3(1+a_{t-1})} \ge \frac{1}{4},\quad
\frac{B_t(1+a_{t-1})}{\theta^c-1} \le 1,\\
4\theta \exp \left ( -\frac{\theta^{2\epsilon}}4 \right )
\le \exp \left ( -\frac{\theta^{2\epsilon}}5 \right ),\quad
\theta^{t} \exp \left ( -\frac1 4 
\theta^{\epsilon(t+1)} \right ) \le \frac{12}{\pi^2 t^2}
\theta \exp \left ( -\frac{\theta^{2\epsilon}}4 \right ).
\nonumber
\end{align}
Then,
\begin{align}
\P\left (
\left \vert \frac{\X_t}{\theta^t} - 1\right \vert \le 
2\theta^{-(1-2\epsilon)/2}
\text{ for all } t
\right )
\ge 1 -  \exp\left ( -\frac{\theta^{2\epsilon}}{5} \right ).
\label{Eq:tight2}
\end{align}
\end{lemma}
\begin{proof}
Let $A_t :=\{ (1-a_t) \theta^t \le \X_t \le (1+a_t) \theta^t \}$ and
$A := \bigcap_{t=0}^\infty A_t$.
Abbreviating 
$x_1 := (1-a_{t-1}) \theta^{t-1}$ and
$x_2 := (1+a_{t-1}) \theta^{t-1}$, we can write
$$
\P(A_t \vert A_{t-1})=\P\left (b_t x_1 \theta \le \X_t \le B_t x_2 \theta \vert x_1 \le \X_{t-1} \le x_2 \right ).$$
By Lemma~\ref{Cor:bBsum}, we have
\begin{align*}
\P(A_t \vert A_{t-1})
\ge 1 &- (1-a_{t-1})\theta^t 
\exp \left ( - \frac{(1-\theta^{-c})^2}{2(1-a_{t-1})} \theta^{\epsilon(t+1)} \right ) \\
&- \frac{B_t(1+a_{t-1})}{\theta^c-1} \theta^{ct - (1-2\epsilon)/2}
\exp \left ( - \frac{(1-a_{t-1})(1-\theta^{-c})^2}{3(1+a_{t-1})} \theta^{\epsilon(t+1)} \right )\\
&\ge 1 - 2 \theta^t \exp \left ( -\frac{1}{4} \theta^{\epsilon(t+1)} \right ),
\end{align*}
where we have used \eqref{Eq:Bbcond}.
Using the last condition of \eqref{Eq:Bbcond}, we have
$$
\sum_{t=1}^\infty \theta^t \exp \left ( -\frac{1}{4} \theta^{\epsilon(t+1)} \right )
\le \frac{12}{\pi^2}\theta \exp\left ( -\frac{\theta^{2\epsilon}}{4} \right )
\sum_{t=1}^\infty \frac{1}{t^2}
=2 \theta \exp\left ( -\frac{\theta^{2\epsilon}}{4} \right ).
$$
Hence, 
$$
\P(A) 
\ge 1 - 2 \sum_{t=1}^\infty \theta^t \exp \left ( -\frac{\theta^{\epsilon(t+1)}}{4} \right )
\ge 1 - 4 \theta \exp\left ( -\frac{\theta^{2\epsilon}}{4} \right )
\ge 1 - \exp\left ( -\frac{\theta^{2\epsilon}}{5} \right ),
$$
where we have used Lemma~\ref{Lem:Markov}.
Since $1-\theta^{-ct} \le 2$ and, therefore,
$a_t \le 2 \theta^{-(1-2\epsilon)/2}$, the probability in \eqref{Eq:tight2} is larger than $\P(A)$ and the proof is completed.
\end{proof}

\noindent
{\bf Definition.}
By $\theta_0(\epsilon)$ we denote the infimum over all $\theta$ that satisfy \eqref{Eq:Bbcond}.

\begin{remark}
\label{Rem:theta} 
If we are given a weaker condition
in Lemma~\ref{Lem:strict} such that 
there are $x$ and $y$ such that $x \ge \theta \ge y> \theta_0(\epsilon)$,
then we have
$$
\P\left (\left . \left \vert \frac{\X_t}{\theta^t} - 1 \right \vert \le 2 \theta^{-(1-2\epsilon)/2}
\text{ for all } t\right \vert y \le \theta \le x \right ) \ge 1 - \exp \left (-\frac{y^{2\epsilon}}5\right ).
$$
\end{remark}
\begin{lemma}
\label{Lem:Olim}
For a discrete-time stochastic process $Z_t$ and a nonzero function $f(t)$,
define
\begin{align*}
J:= \left \{ \lim_{t\rightarrow\infty} \frac{Z_t}{f(t)} = 1 \right \},\quad
D_{m,k}:=\left \{ \left \vert \frac{Z_k}{f(k)}-1 \right \vert  \le 2^{-m}\right \},
\\
O_{m,\tau}:=\bigcap_{k=\tau}^\infty D_{m,k},
\quad O_{m}:=\bigcup_{\tau=1}^\infty O_{m,\tau},\quad O = \bigcap_{m=1}^\infty O_m.
\end{align*}
Then, $O=J$ and
\begin{align*}
\lim_{\tau\rightarrow\infty}\P(O_{m,\tau}) \ge \P(J),
\end{align*} 
for any positive integer $m$.
\end{lemma}
\begin{proof}
First note that
$O_{m,\tau} \subset O_{m,\tau+1} \subset O_m$ and $O\subset O_{m+1} \subset O_{m}$.
Consider any outcome $\omega \in J$ and fix $m$. Under $\omega$, for any $0<\epsilon'\le 2^{-m}$ there is
$k_0$ such that $\vert Z_k/f(k)-1 \vert \le \epsilon' \le 2^{-m}$ for all $k \ge k_0$,
which implies $\omega \in O_m$. 
Since $m$ is arbitrary, we have $J \subset O$.
\medskip

Now consider $\omega' \notin J$. Then under $\omega'$ there is $\epsilon'>0$ such that 
$\vert Z_k/f(k)-1 \vert  >\epsilon'$ for infinitely many $k$'s. Hence,
$\omega'$ cannot be an outcome in $O_m$ if $2^{-m}<\epsilon'$. Hence, $\omega' \notin O$ and,
accordingly, $O \subset J$.
Even if $J$ is empty, the proof of $O\subset J$ is still applicable 
and the rest of the statement is trivially valid.\medskip

Since  $O\subset O_m$ and
$\P(O_m) = \lim_{\tau\rightarrow\infty} \P(O_{m,\tau})$ for any $m$,
the proof is completed.
\end{proof}
\begin{lemma}
\label{Thm:efd}
If $G$ is of the type I with $n=1$ or with $n=2 \text{ and }\alpha<1$, then
let
\begin{align}
\label{Eq:IJ}
J:=\left \{ \lim_{t\rightarrow\infty} \frac{W_t}{u_n(t)} = 1 \right \}.
\end{align}
If $G$ is of the type II, then let
$$
J:=\left \{ \lim_{t\rightarrow\infty} \frac{\lo 2{W_t}}{\log t} = \frac{1}{\alpha}
\right \}
$$
for $n=1$,
$$
J:=\left \{ \lim_{t\rightarrow\infty} \frac{\lo 3{W_t}}{\log t} = \frac{1}{1+\alpha}
\right \}
$$
for $n=2$,
and
$$
J:=\left \{ \lim_{t\rightarrow\infty} \frac{1}{\lo{n-1} t} \log \left ( 
\frac{\lo 2 {W_t}}{ t}\right ) = -\alpha
\right \}
$$
for $n\ge 3$.
For the type II tail function or for the type I tail function with $n=1$, fix an arbitrary $\epsilon$ satisfying $0<\epsilon<1/2$.
For the type I tail function with $n=2$ and $\alpha<1$, fix an arbitrary $\epsilon$ satisfying
$\alpha < 2\epsilon < 1$. 
Abbreviate $\theta_k:=(1-\beta)W_k$ and
let
\begin{align*}
C_k &:= \bigcap_{t=k}^\infty \left \{ \left \vert \frac{N_k(t)}{\theta_k^{t-k}} -1 \right \vert \le 
2\theta_k^{-(1-2\epsilon)/2}
\right \},\quad
E_\tau := \bigcap_{k=\tau}^\infty C_k, \quad E := \bigcup_{\tau=1}^\infty E_\tau,
\end{align*}
where we assume $N_k(t)/\theta_k^{t-k} = 1$  and $\theta_k^{-(1-2\epsilon)/2} = \infty$
if $W_k = 0$.
Then $\P(E\vert J) = 1$.
\end{lemma}
\begin{proof}
Let $ U(x,m):= (1-2^{-m}) u_n(x) $
for the type I tail function and
\begin{align*}
U(x,m) &:= \exp \left ( x^{(1-2^{-m})/\alpha} \right ),\\
U(x,m) &:= \en{2}{x^{(1-2^{-m})/(1+\alpha)}},\\
U(x,m) &:= \en{2}{ x \exp\left (  - \alpha ( 1 + 2^{-m} ) \lo{n-1}{x} \right )},
\end{align*}
for the type II tail function with $n=1$, $n=2$, and $n\ge 3$, respectively.
In the above definition, $x$ is assumed sufficiently large that $U(x,m)$ is well defined.
With the fixed $\epsilon$, for any positive $m$ there is $\tau_0(m)$ such that
\begin{align}
\label{Eq:tau0}
\exp \left ( - \frac{(1-\beta)^{2\epsilon}}{5}  U(t,m)^{2\epsilon} \right ) \le \frac{1}{t(t+1)},
\end{align}
for all $t \ge \tau_0(m)$.
Let
\begin{align*}
D_{m,k}&:=\left \{ \left \vert \frac{W_k}{u_n(k)}-1 \right \vert  \le 2^{-m}\right \},
\end{align*}
for the type I tail function and
\begin{align*}
D_{m,k}&:=\left \{ \left \vert \frac{\alpha \lo 2 {W_k}}{\log k}-1 \right \vert  \le 2^{-m}\right \},\\
D_{m,k}&:=\left \{ \left \vert \frac{(1+\alpha) \lo 3 {W_k}}{\log k}-1 \right \vert  \le 2^{-m}\right \},\\
D_{m,k}&:=\left \{ \left \vert \frac{1}{\alpha\lo{n-1} k}\log \left ( \frac{\lo 2 {W_k}}{\log k}\right ) +1 \right \vert  \le 2^{-m}\right \},
\end{align*}
for the type II tail function with $n=1$, $n=2$, and $n\ge 3$, respectively.
Let
\begin{align*}
C_k^c &:= \bigcup_{t=k}^\infty \left \{ 
\left \vert \frac{N_k(t)}{\theta_k^{t-k}} -1 \right \vert > 
2\theta_k^{-(1-2\epsilon)/2}  \right \},\quad
E_\tau^c := \bigcup_{k=\tau}^\infty C_k^c,\\
O_{m,\tau}&:=\bigcap_{k=\tau}^\infty D_{m,k},
\quad O_{m}:=\bigcup_{\tau=1}^\infty O_{m,\tau},
\end{align*}
where $\tau$ is assumed large enough such that $u_n(\tau)$, $\log \tau$,
and $\lo{n-1}{\tau}$ are well defined.
Note that, for all sufficiently large $\tau$,
$E_\tau \subset E_{\tau+1} \subset E$.
Now, consider $\P(E_\tau^c \cap O_{m,\tau})$ for $m\ge 1$.
By the sub-additivity of the probability measure, we have
\begin{align*}
\P(E_\tau^c \cap O_{m,\tau}) 
&= \P\left ( \bigcup_{k=\tau}^\infty \left ( C_k^c \cap O_{m,\tau} \right )\right )
\le \sum_{k=\tau}^\infty \P\left (C_k^c \cap O_{m,\tau} \right )
\le \sum_{k=\tau}^\infty \P(C_k^c \cap D_{m,k} ),
\end{align*}
where we have used $O_{m,\tau} \subset D_{m,k}$ for any $k\ge \tau$. 
Now fix an integer $m \ge 1$ and consider large enough $\tau$ such that $\tau > \tau_0(m)$ 
as in \eqref{Eq:tau0} and
$(1-\beta)U(k,m) > \theta_0(\epsilon)$ for all $k \ge \tau$.
Since $\P(C_k^c \cap D_{m,k}) \le \P(C_k^c \vert D_{m,k}) = 1 - \P(C_k \vert D_{m,k})$,
Remark~\ref{Rem:theta} with
$y \mapsto (1-\beta) U(k,m)$ gives
$$
\P(C_{k+\tau}^c\cap D_{m,{k+\tau}}) \le \frac{1}{(k+\tau)(k+\tau+1)},
$$
for all $k \ge 0$, where we have used \eqref{Eq:tau0}.
Therefore, we have
\begin{align}
\lim_{\tau\rightarrow\infty} \P(E_\tau^c \cap O_{m,\tau} ) &\le
\lim_{\tau\rightarrow\infty} 
\sum_{k=0}^\infty \P(C_{\tau+k}^c \cap D_{m,\tau+k}) \le
\lim_{\tau\rightarrow\infty} \frac{1}\tau = 0.
\label{Eq:Etauc}
\end{align}
Since $\P(O_{m,\tau}) \le p_s$ for all sufficiently large $\tau$, $\P(J)=p_s$, and
$
\P(E_\tau\cap O_{m,\tau}) = \P(O_{m,\tau}) - \P(E_\tau^c \cap O_{m,\tau}),
$
Lemma~\ref{Lem:Olim} gives
\begin{align}
\label{Eq:Etau}
\lim_{\tau\rightarrow\infty} \P(E_\tau \cap O_{m,\tau}) 
=\lim_{\tau\rightarrow\infty} \P(O_{m,\tau}) 
=p_s.
\end{align}
Since $E_\tau \subset E$ and $O_{m,\tau} \subset O_m$, we have
$
p_s \ge \P(E \cap O_m) \ge 
\lim_{\tau\rightarrow\infty} \P(E_\tau \cap O_{m,\tau}) 
=p_s,
$
for all~$m$.
Therefore,
$
\P(E\vert J) \P(J)
=\P(E\cap J) 
=\lim_{m\rightarrow\infty} \P(E\cap O_m) = p_s.
$
Since  $\P(J)=p_s$, the proof is completed.
\end{proof}
\begin{remark}
\label{Rem:um}
In the proof, \eqref{Eq:tau0} plays the decisive role.
If $G$ is of type I with $n \ge 3$ or with $n=2$ and $\alpha > 1$,
\eqref{Eq:tau0} is not applicable. Within the tools we are equipped with,
we are not aware of a similar result to Lemma~\ref{Thm:efd}
for fast decaying tail functions.
\end{remark}

\begin{remark}
We can rewrite Lemma~\ref{Thm:efd} as follows. 
For any type II tail function and for a type I tail function with
$n=1$ or with $n=2$ and $\alpha<1$, for any $\epsilon>0$
$$
\lim_{\tau\rightarrow\infty} 
\P \left ( \left \vert \frac{N_k(k+s)}{(1-\beta)^s W_k^s} - 1 \right \vert 
< \epsilon \text{ for all } s\ge 0 \text{ and for all }k \ge \tau \right ) = p_s.
$$
\end{remark}
\begin{remark}
If $G$ is of the Fr\'echet type in \cite{Park2023}, then setting
$$
J:=\left \{ \lim_{t\rightarrow\infty} \frac{\lo 2{W_t}}{t} = \nu(\alpha)
\right \}
$$
in Lemma~\ref{Thm:efd} gives $\P(E\vert J)=1$.
\end{remark}

The following two lemmas, which will not be directly used later, are for explaining at what point the proof of Theorem~\ref{Th:efd} becomes difficult and  also for providing a more compelling rationale of introducing \dfm and \sfm.
\begin{lemma}
\label{Lem:XXi}
For the \fm, define two random sequences $(\eta_t)$ and $(\xi_t)$ as
$$
X(t) = (1-\beta)\Xi(t) + \eta_t \Xi(t)^{3/4} ,\quad
(1-\beta) \Xi(t) = X(t) + \xi_t X(t)^{3/4} .
$$
In case $X(t) =\Xi(t) = 0$, we define $\xi_t = \eta_t=0$.
Then almost surely
$$
\lim_{t\rightarrow\infty} \eta_t=
\lim_{t\rightarrow\infty} \xi_t=0.
$$
\end{lemma}
\begin{proof}
Let
\begin{align*}
A_t &:= \left \{ \left \vert(1-\beta) \Xi(t) - X(t)  \right \vert \le X(t)^{2/3} \right \} ,\quad B_t := \bigcap_{k=t}^\infty A_k,\quad B := \bigcup_{t=1}^\infty B_t\\
C_t &:= \left \{ \left \vert X(t) - (1-\beta) \Xi(t) \right \vert \le \tfrac{1}{3} (1-\beta)^{2/3}\Xi(t)^{2/3} \right \},\quad
D_t := \bigcap_{k=t}^\infty C_k,\quad D := \bigcup_{t=1}^\infty D_t,\\
E_t&:= \left \{ (1-\beta) \Xi(t) > (t+1)^3 (t+2)^3\right \},\quad
J_t := \bigcap_{k=t}^\infty E_k,\quad J := \bigcup_{t=1}^\infty J_t.
\end{align*}
By Theorems~\ref{Th:mainthm} and \ref{Th:mainII}, we have $\P(J) = p_s$.
For positive $z$, let $y_1(z)$ and $y_2(z)$ be the (unique) positive solution of the equations
\smash{$
z = y_1 + y_1^{_{2/3}}$} and \smash{$z = y_2 - y_2^{_{2/3}}$}, respectively.
Using $y_1$ and $y_2$, we write 
$$
A_t \cap E_t = \left \{  y_1\left ( (1-\beta) \Xi(t) \right ) 
\le
X(t) \le y_2\left ( (1-\beta) \Xi(t) \right )
\right \} \cap E_t.
$$
Note that if $z > 2$ we have $y_1 < z - z^{2/3}/3 < z+ z^{2/3}/3 < y_2$. 
Hence,
$
C_t \cap E_t
\subset A_t \cap  E_t
$
and, accordingly, $D_t \cap J_t \subset B_t \cap J_t$.
By Chebyshev's inequality we have 
$$
\P\left (C_t \vert \Xi(t) \right ) \ge 1-9 \frac{\beta}{(1-\beta)^{1/3}}\Xi(t)^{-1/3},
$$
which gives
\begin{align}
\P(C_t^c \vert E_t) \le  \frac{9\beta}{(t+1)(t+2)}.
\label{Eq:XxiCcEt}
\end{align}
For $t\ge \tau$, we have
$$
\P(D_t^c \cap J_\tau) \le \sum_{k=t}^\infty \P(C_k^c \cap J_\tau)
\le \sum_{k=t}^\infty \P(C_k^c \cap E_k) 
\le \sum_{k=t}^\infty \P(C_k^c \vert E_k) 
\le \sum_{k=t}^\infty \frac{9\beta}{(k+1)(k+2)}=
\frac{9\beta}{t+1},
$$
where used the definition of $D_t$ for the
first inequality, $J_\tau \subset E_k$ for the second inequality,
$\P(E_k) \le 1$ for the third inequality, and \eqref{Eq:XxiCcEt} for the last
inequality.
Therefore, for any $\tau$, we have
$
\P(D^c \cap J_\tau) = \lim_{t\rightarrow\infty} \P(D_t^c \cap J_\tau) = 0
$
and, accordingly, $\P(D \cap J) = \P(B\cap J) = p_s$.
As $ x^{2/3}=x^{3/4} x^{-1/12}$ and $X(t)$ diverges almost surely on survival,
the proof is completed.
\end{proof}
\begin{remark}
Lemma~\ref{Lem:XXi} is applicable even if the support of $\mu$ is bounded 
because $\Xi(t)$ grows at least exponentially on survival.
Therefore, regardless of the type of $G$ we have 
almost surely on survival $X(t) \sim (1-\beta)\Xi(t)$ and the
relative error of the approximation 
is at most $\Xi(t)^{-1/4}$.
\end{remark}

\noindent
{\bf Definition (for Lemma~\ref{Lem:XigivesW}).}
We define positive sequences $(a_t)_{t\ge 1}$ and $(y_t)_{t\ge 1}$ as
\begin{align*}
a_t := \frac{1}{\log (t+2)} - \frac{1}{\log(t+3)}, \quad
y_t := \frac{\log a_t}{\log (1-\beta)}.
\end{align*}
Note that $a_t$ is monotonically decreasing with
$1/a_t \sim t (\log t )^2$ and $y_t$ is monotonically increasing.
For $Y > y_t$, we define
\begin{align*}
\Wl(Y,t)&:= \inf \left \{ x>0 : 1-\left [ 1-\beta G(x)\right ]^Y \le a_t \right \},\\
\Ws(Y,t)&:= \sup \left \{ x>0 : \left [ 1-\beta  G(x)\right ]^Y\le a_t \right \} - \epsilon,
\end{align*}
where $\epsilon$ is an arbitrary small positive number.
Since $a_t < \frac{1}{2}$, we have $\Wl(Y,t) > \Ws(Y,t)$.
For $Y \le y_t$ we define $\Wl(Y,t) = \Ws(Y,t)=0$.

\begin{remark}
Since $G(x)$ is a right-continuous-left-limit function,
the purpose of introducing $\epsilon$ is to guarantee  
$\left ( 1-\beta  G(\Ws) \right )^Y \le a_t$.
Without $\epsilon$, $\left ( 1-\beta  G(\Ws) \right )^Y$ may be larger than
$a_t$.
In the case that $G(x)$ is a strictly decreasing continuous function, we rather define, for $Y > y_t$,
\begin{align*}
1-\beta G\left (\Wl(Y,t) \right ) &= (1-a_t)^{1/Y},\\
1-\beta G\left (\Ws(Y,t) \right ) &= (a_t)^{1/Y}.
\end{align*}
\end{remark}
\begin{lemma}
\label{Lem:XigivesW}
For the type I tail function, define
\begin{align*}
J:=\left \{ \lim_{t\rightarrow\infty} \frac{\log \Xi(t)}{t \logn t} = \frac{1}{\alpha} \right \}.
\end{align*}
For the type II tail function, define
\begin{align*}
J:= \left \{ \lim_{t\to\infty} \frac{\lo 2{\Xi(t)}}{\log t}=1+\frac1\alpha
\right \},
\end{align*}
for $n=1$, 
\begin{align*}
J:= \left \{\lim_{t\to\infty} \frac{\lo 3 {\Xi(t)}}{\log t}
=\frac1{1+\alpha}\right \},
\end{align*}
for $n=2$, and
\begin{align*}
J:=\left \{ \lim_{t\to\infty} \frac{1}{\lo{n-1} t}\log\left ( \frac{\lo 2 {\Xi(t)}}{t}\right )=-\alpha \right \},
\end{align*}
for $n \ge 3$. We have proved $\P(J)=p_s$.
Let
\begin{align*}
A_t &:= \{ \Ws(\Xi(t),t) < W_t \le \Wl(\Xi(t),t)\},\quad B_t := \bigcap_{k=t}^\infty A_k,\quad
B := \bigcup_{t=1}^\infty B_t.
\end{align*}
Then, $\P(B \cap J) = p_s$.
\end{lemma}
\begin{proof}
Let $C_t := \{ \Xi(t) > y_t\}$, $D_\tau := \bigcap_{k=\tau}^\infty C_k$,
and $D := \bigcup_{\tau=1}^\infty D_\tau$.
Since $J \subset D$, it is enough to prove $\P(B\cap D) = p_s$, because
$\P(B\cap J) = \P(B\cap D) - \P (B \cap (D \setminus J))$ and
$\P(D\setminus J)=0$.
For any integer $Y > y_t$, we have
\begin{align*}
\P(A_t^c \vert \Xi(t)=Y ) 
&= \P(W_t \le \Ws \vert Y )
+ 1-\P(W_t \le \Wl \vert Y )\\
&= \left (1-\beta G(\Ws) \right )^{Y}
+ 1 - \left (1-\beta G(\Wl) \right )^{Y}
\le 2 a_t,
\end{align*}
where we have used Lemma~\ref{Th:Wa}.
Accordingly,
$
\P(A_t^c \vert C_t ) \le 2 a_t.
$
If $t \ge \tau$, then we have
$$
\P(B_t^c \cap D_\tau) \le \sum_{k=t}^\infty \P(A_k^c \cap D_\tau)
\le \sum_{k=t}^\infty \P(A_k^c \cap C_k) 
\le \sum_{k=t}^\infty \P(A_k^c \vert C_k) \le \sum_{k=t}^\infty 2 a_k
= \frac{2}{\log(t+2)}.
$$
Therefore, for any $\tau$, we have
$
\P(B^c \cap D_\tau) = \lim_{t\rightarrow\infty} \P(B_t^c \cap D_\tau) = 0
$
and so $\P(B^c \cap D) = 0$ and the proof is completed.
\end{proof}

\begin{remark}
Assume $G$ is a strictly decreasing continuous function.
Let $w_0(Y,t)$ and $w_1(Y,t)$ be the solution of 
\begin{align*}
-\log G(w_1) &= \log Y + \log \tfrac{2\beta}{a_t},\\
-\log G(w_0) &= \log Y +\log \beta + \lo 2 {\tfrac{1}{a_t}}.
\end{align*}
For sufficiently large $t$ and $Y$,
we have $1-(1-a_t)^{1/Y} > a_t/(2 Y)$ 
and $1-a_t^{1/Y} < - (\log a_t)/Y$,
which gives $G(w_1)<G(\Wl) < G(\Ws) < G(w_0)$.
Since $G$ is a decreasing function, we have
$w_0 \le \Ws \le \Wl \le w_1$.
Hence, even if we define $A_t$ by the condition $w_0 \le W_t \le w_1$
with $Y=\Xi(t)$,
Lemma~\ref{Lem:XigivesW} remains valid. \medskip

Now consider $G(x) = \exp(-x^\alpha)$, 
which entails $\alpha W_t^\alpha\sim t \log t$.
Then, we have
\begin{align*}
w_1 &= \left [ \log \Xi(t) \right ]^{1/\alpha}
\left [ 1 + \frac{\log (2 \beta/a_t)}{\log \Xi(t)} \right ]^{1/\alpha},\\
w_0 &= \left [ \log \Xi(t) \right ]^{1/\alpha}
\left [ 1 + \frac{\log \beta + \lo 2{1/a_t}}{\log \Xi(t)} \right ]^{1/\alpha}.
\end{align*}
We define a random sequence $(b_t)_{t\ge 1}$ by
$
W_t = \left [ \log X(t) \right ]^{1/\alpha} 
\exp(b_t/t).
$
Then, the above discussion together with  Lemma~\ref{Lem:XXi} shows that,
almost surely on survival,
$$
0 \leq 
\liminf_{t\rightarrow\infty} b_t \leq
\limsup_{t\rightarrow\infty} b_t \le 1.
$$
Writing $W_t = u_1(t) e^{c_t}$ we see from Theorem~\ref{Th:mainthm} that, almost surely on survival,
$
\lim\limits_{t\rightarrow\infty} c_t = 0.
$
Hence, 
\begin{align}
\label{Eq:XW_str}
\log X(t) = u_1(t)^\alpha \exp \left ( \alpha c_t - \alpha \tfrac{b_t}{t} \right ),\quad W_t = u_1(t) e^{c_t}.
\end{align}
Unfortunately, \eqref{Eq:XW_str} does not give an accurate estimate of $X(t)$, because
$$
X(t) \exp ( -u_1(t)^\alpha) \ge \exp \left ( \alpha u_1(t)^\alpha c_t - 
\alpha u_1(t)^\alpha \tfrac{b_t}{t} \right )
$$
and it is unclear whether the right hand side will diverge or not.
Since $X(t)$ rather than $\log X(t)$ is necessary to study the \efd and we do not have a tool to tame $c_t$ and $b_t$, \eqref{Eq:XW_str} is too coarse to give a proof for Theorem~\ref{Th:efd} even for this special $G$. So we are forced to introduce simplified models, the \dfm and the \sfm, to prove variants of Theorem~\ref{Th:efd}.
\end{remark}

\subsection{\label{Sec:deter}Deterministic \fm and its \efd}
Lemma~\ref{Thm:efd} shows that if $G$ is of type I with
$n=1$ or with $n=2$ and $\alpha<1$ then 
almost surely on survival
$$
\lim_{t \rightarrow\infty} \frac{W_t}{u_n(t)} 
= \lim_{k\rightarrow\infty} \frac{N_k(k+s)}{(1-\beta)^{s} W_k^{s}} = 1,
$$
for any nonnegative integer $s$.
Hence, setting $W_k = u_n(k)$ and $N_k(t) = (1-\beta)^{t-k}u_n(k)^{t-k}$ for all large $k$ and $t\ge k$ gives a good approximation of the models on survival. This approximation is especially convenient for the \fm. 
In this context, we are motivated to introduce
the deterministic \fm as follows.
\medskip

\noindent
{\bf Definition of the \dfm.}
At each generation $k> 0$ a new mutant with fitness $W_k= u_n(k)$ appears. 
In case $u_n(k)$ is ill defined, we set $W_k =1/(1-\beta)$.
The number of non-mutated descendants of $W_k$ grows deterministically as
\begin{align}
\mkd(t) := (1-\beta)^{t-k}W_k^{t-k}.
\label{Eq:Nkd}
\end{align}
where we neglect not only stochasticity but also the error due to the discreteness of $N_k(t)$.
Note that in the \dfm only  type I tail functions are under consideration and we make no restriction on $n$ and $\alpha$.
\bigskip

Notice that we added the superscript $\mathrm{D}$ in $\mkd$ to 
discern them from their stochastic counterparts.
Since no fluctuation is present, the limit behavior of the \efd for the \dfm 
becomes a problem of calculus.
In what follows, we will find a limit theorem of the \efd for the \dfm.
To this end, we begin with the following elementary lemma.

\begin{lemma}
\label{Lem:rudiment}
Assume $f$ is a positive continuous function that has a unique local maximum at $x_c$ 
in a domain $[a-1,b+1]$, where $a,b$ are integers and $x_c$ need not be an integer. That is, $f(x) < f(y)$ if $a-1\le x< y\le x_c$ and
$f(x) > f(y)$ if $ x_c \le x < y \le b+1$. Assume $a<x_c < b$.
Let 
$$
F(x) := \sum_{k=a}^{\lfloor x \rfloor} f(k).
$$
Then, for any $a\le x\le b$,
$$
\left \vert F(x)- \int_a^x f(y) dy \right \vert \le 7 f(x_c).
$$
\end{lemma}
\begin{proof}
Define
$
f_-(x) := f(\lfloor x \rfloor)$ and $ f_+(x) = f(\lceil x \rceil)$.
Then, for $a \le x \le b$,
$$
F(x) = \int_a^{\lfloor x \rfloor+1} f_-(y) dy = \int_{a-1}^{\lfloor x \rfloor} f_+(y) dy.
$$
Note that $f_-(x) \le f(x) \le f_+(x)$ if $x < \lfloor x_c\rfloor$ and $f_-(x) \ge f(x) \ge f_+(x)$
if $x \ge \lceil x_c \rceil$.
We abbreviate $m_c := \lfloor x_c \rfloor$ and $m := \lfloor x \rfloor$.
If $m< m_c$, then
$$
F(x) = \int_{a-1}^m f_+(y) dy \ge 
\int_{a-1}^a f(y) dy+\int_{a}^x f(y) dy - \int_{m}^x f(y) dy
$$
and
$$
F(x) = \int_a^{m+1} f_-(y) dy 
\le \int_a^x f(y) dy + \int_x^{m+1} f(y),
$$
which gives
$$
\left \vert F(x)- \int_a^x f(y) dy \right \vert \le  2 f(x_c).
$$
For $m=m_c$, we use
$$
\int_{a-1}^a f(y) dy + \int_a^{x} f(y) dy -\int_{m-1}^{x} f(y) dy\le F(m-1)
\le \int_{a}^{x} f(y)dy - \int_{m}^x f(y) dy
$$
and $F(x) = F(m-1) + f(m)$, to get
$
\left \vert F(x)- \int_a^x f(y) dy \right \vert \le  4 f(x_c).
$
If $m > m_c$, we consider
$$
F(x) - F(x_c) =\sum_{k=m_c+1}^m f(k) = \int_{m_c+1}^{m+1} f_-(y) dy = f(m_c+1) + \int_{m_c+1}^m f_+(y) dy,
$$
which gives
\begin{align*}
F(x)-F(x_c) &\ge \int_{x_c}^x f(y) dy - \int_{x_c}^{m_c+1} f(y) + \int_{x}^{m+1} f(y) dy,\\
F(x) - F(x_c) &\le f(x_c) +\int_{x_c}^x f(y) dy - \int_{x_c}^{m_c+1} f(y) - \int_{m}^x f(y) dy,
\end{align*}
and, therefore,
$
\left \vert F(x) - F(x_c) - \int_{x_c}^x f(y) dy \right \vert \le  3 f(x_c).
$
Since
\begin{align*}
\left \vert F(x) - \int_a^x f(y) dy \right \vert 
&\le \left \vert F(x_c) - \int_{a}^{x_c} f(y) dy \right \vert 
+ 
\left \vert F(x) - F(x_c) - \int_{x_c}^x f(y) dy \right \vert 
\le  7 f(x_c),
\end{align*}
we have the desired result for any $a\le x\le b$.
\end{proof}

\noindent
{\bf Definition.}
We define 
\begin{align*}
H(x,t) &:= \left [ (1-\beta) u_n(x) \right ]^{t-x},\quad
h(x,t):=\log H(x,t),\\
\omega_1(x) &:= (1-\beta) \omega_W \left ( \lo{n-1}{x} \right )
,\\
L_j(x) &:= (-1)^{j-1} \left ( \frac{d}{dx} 
\right )^j \logn{x},\quad
\left . \Omega_j(x) := (-1)^{j-1}\left ( y \frac{d}{dy} \right )^j
\log \omega_W(y) \right \vert_{y=\lo{n-1}{x}},
\end{align*}
where $x \le t$ and $x$ is assumed large enough so that the above definition makes sense.
Note that $(1-\beta)u_n(x) = (\lo{n-1}{x})^{1/\alpha} \omega_1(x)$ and
$\mkd(t)=H(k,t)$.
Also note that
$$
\frac{d}{dx}\log \omega_1(x) = L_1(x) \Omega_1(x), \quad
\frac{d}{dx} \Omega_j(x) = - L_1(x) \Omega_{j+1}(x).
$$
We assume  
$
\lim_{x\rightarrow\infty} \Omega_j(x) = 0
$ for any integer $j \ge 1$; see {\bf (A4)}.

\begin{lemma}
\label{Lem:x0t0}
There are $x_0$ and $t_0$ such that 
\begin{align}
\frac{\partial^2 h(x,t)}{\partial x^2}<0,\quad
\frac{\partial^3 h(x,t)}{\partial x^3}>0,\quad
\frac{\partial^4 h(x,t)}{\partial x^4}<0,
\label{Eq:wcon}
\end{align}
for all $x \ge x_0-1$ and for all $t \ge t_0\ge x_0-1$ with $x\le t$.
\end{lemma}
\begin{proof}
First observe that
\begin{align}
\nonumber
L_1(x) &= \left ( \prod_{k=0}^{n-1} \lo{k}x\right)^{-1},\\
L_j(x) &\sim \frac{(j-1)!L_1(x)}{x^{j-1}} = \frac{(j-1)!}{x^j}  \left ( \prod_{k=1}^{n-1} \lo{k}x\right)^{-1}, \quad \frac{L_j(x)}{L_{j+1}(x)} \sim \frac{x}j,
\label{Eq:Li}
\end{align}
where we use the convention $\prod_{k=1}^0 :=1$. We define
\begin{align*}
\phi_1(x,t)&:=
1 + \alpha\Omega_1(x) + \alpha \frac{L_1^2}{L_2}\Omega_2(x)   + \frac{x}{t}
\left [ 
\left ( 1 + \alpha\Omega_1(x) \right ) 
\left ( \frac{2L_1}{L_2 x} - 1 \right ) - \alpha \frac{L_1^2}{L_2} \Omega_2(x)
\right ],\\
\phi_2(x,t)&:= \phi_1(x,t) - \frac{L_2}{L_3}\frac{\partial \phi_1}{\partial x},\qquad
\phi_3(x,t) := \phi_2(x,t) - \frac{L_3}{L_4}\frac{\partial \phi_2}{\partial x}.
\end{align*}
We write down the derivatives
\begin{equation}
\label{Eq:dhdx}
\begin{aligned}
&\frac{\partial h}{\partial x} = - \frac{1}{\alpha} \logn{x} -
\log \omega_1 (x) + (t-x) L_1(x) \left ( \frac{1}{\alpha} + \Omega_1(x) \right ), 
\frac{\partial^2 h}{\partial x^2} 
= -\frac{t}\alpha L_2(x) \phi_1(x,t),\\
&\frac{\partial^3 h}{\partial x^3}
= \frac{t}{\alpha } L_3(x) \phi_2(x,t), \qquad
\frac{\partial^4 h}{\partial x^4}
= -\frac{t}{\alpha} L_4(x) \phi_3(x,t).
\end{aligned}
\end{equation}
As, by \eqref{Eq:Li}, $\phi_j(x,t)$ is positive for all sufficiently large $x$ and $t$ existence of $x_0$ and~$t_0$
follows.
\end{proof}

\begin{remark}
We fix such $x_0$ and $t_0$ in the following and
treat $x_0$ as the initial generation and we consider
only $t \ge t_0$.
\end{remark}
\begin{lemma}
\label{Lem:xc_at}
Let $x_c(t)$ be the location of the maximum of $h(x,t)$ for given $t$
and let 
\begin{align*}
\kappa_t :=& -\left . \frac{\partial^2 h}{\partial x^2} \right \vert_{x=x_c},\quad
d_t :=\frac{1}{3!} \left . \frac{\partial^3 h}{\partial x^3} \right \vert_{x=x_c}.
\end{align*}
Then,
\begin{align}
\label{Eq:xc}
x_c &\sim t  \prod_{k=1}^n \frac1{\lo{k}{t}},\\
\label{Eq:atdt}
\kappa_t &\sim \frac{\lo{n}{t}}{\alpha t} 
\prod_{k=1}^n \lo{k}{t} \sim \frac{\logn{t}}{\alpha x_c},\quad
d_t \sim \frac{\logn{t}}{3 \alpha t^2}
\prod_{k=1}^n \left ( \lo{k}{t}\right )^2.
\end{align}
\end{lemma}
\begin{proof}
From \eqref{Eq:dhdx} and \eqref{Eq:wcon}, we have
\begin{align*}
0 
= - \frac{1}{\alpha} \logn{x_c} -
\log \omega_1 (x_c)
+ (t-x_c) L_1(x_c) \left ( \frac{1}{\alpha}
+ \Omega_1(x_c) \right ),
\end{align*}
for given $t$. Obviously, the solution of the equation diverges
with $t$, so 
$x_c$ satisfies
$$
t \sim \frac{\logn {x_c} }{ L_1(x_c)} = x_c \prod_{k=1}^n \lo{k}{x_c}.$$
Therefore,
$$
x_c \sim 
t \prod_{k=1}^n \frac1{\lo{k}{x_c}}
\sim 
t \prod_{k=1}^n \frac1{\lo{k}{t}}.
$$
Considering $\phi_j(x_c)\sim 1$ and using \eqref{Eq:Li}, we get 
the desired result.
\end{proof}
\begin{remark}
In the following, $t_0$ is further assumed so large that
$x_c > x_0$ for all $t > t_0$.
\end{remark}
\begin{lemma}
\label{Lem:hhx}
$$
\left \vert h(x,t) - h(x_c,t) \right \vert 
\le \frac{\kappa_t}{2} (x-x_c)^2 \left ( 1 + \frac{2 d_t}{\kappa_t}\vert x-x_c\vert \right ).
$$
\end{lemma}
\begin{proof}
By~\eqref{Eq:wcon}, we have
\begin{align*}
-\frac{\kappa_t}{2} (x-x_c)^2 + d_t (x-x_c)^3 \le 
h(x,t) -h(x_c,t) \le -\frac{\kappa_t}{2} (x-x_c)^2  
\end{align*}
for $x_0\le x \le x_c$ and
\begin{align}
-\frac{\kappa_t}{2} (x-x_c)^2 + d_t (x-x_c)^3 \ge 
h(x,t) -h(x_c,t) \ge -\frac{\kappa_t}{2} (x-x_c)^2  
\label{Eq:xlxc}
\end{align}
for $x \ge x_c$,
and, therefore, we get the desired result.
\end{proof}
\begin{lemma}
\label{Lem:deter}
We define
\begin{align}
\nonumber
\wX(t) &:= \sum_{k=x_0}^t \mkd(t),\quad
\Phi(f,t) := \frac{1}{\wX(t)}\sum_{k=x_0}^{t} \mkd(t) \Theta(f-u_n(k)),\\
\wS_t &:= \frac{1}{\wX(t)}\sum_{k=x_0}^{t} u_n(k) \mkd(t) ,\quad
\ds_t := \left (\frac{1}{\wX(t)}\sum_{k=x_0}^{t} (u_n(k)-\wS_t)^2 \mkd(t)
\right )^{1/2},
\label{Eq:manyf}
\end{align}
where we only consider $t > t_0$.
Then, 
\begin{align*}
\wS_t \sim v_n(t),\quad \ds_t \sim \s_n(t),\quad
\lim_{t\rightarrow\infty} \Phi(v_n(t) +  \s_n(t) y,t)=
\lim_{t\rightarrow\infty} \Phi(\wS_t +  \ds_ty,t)=\Upsilon(y),
\end{align*}
where
\begin{align}
\label{Eq:vnt}
v_n(t) &:=
\alpha^{-\delta_{n,1}/\alpha} 
\left ( \lo{n-1}t \right )^{1/\alpha}
L \left (\left ( \lo{n-1}t \right )^{1/\alpha} \right )
,\\
\s_n(t) &:= \frac{v_n(t)}{\sqrt{\alpha t}} \left ( \prod_{k=1}^{n-1} \lo{k}{t}
\right )^{-1/2}. 
\label{Eq:snt}
\end{align}
\end{lemma}
\begin{proof}
First note that \eqref{Eq:atdt} gives, for any $0<\epsilon<1$,
$$
\lim_{t\rightarrow\infty} \frac{2d_t}{\kappa_t} \kappa_t^{-(1-\epsilon)/2} = 0,
$$

If $\vert x-x_c\vert  \le \kappa_t^{-(1-\epsilon)/2}$ 
in Lemma~\ref{Lem:hhx} for some $0<\epsilon<1$
and $t$ is sufficiently large that $ 2 d_t \kappa_t^{-(1-\epsilon)/2}/\kappa_t \le 1$, then
$
\left \vert h (x,t) - h(x_c,t) \right \vert  \le \kappa_t^{\epsilon} ,
$
which approaches zero as $t$ goes to infinity; see \eqref{Eq:atdt}. Therefore, 
$h(m,t) \sim h(x_c,t)$, as $t\to\infty$
for $\vert m - x_c \vert \le \kappa_t^{-(1-\epsilon)/2}$, which gives
$$
\lim_{t\rightarrow\infty}  \kappa_t^{(1-\epsilon)/2} \frac{\wX(t)}{H(x_c,t)} = \infty.
$$
for any $\epsilon>0$.
In other words, for any $\epsilon>0$ there is $t_1$ such that
$\wX(t) \ge H(x_c,t) \kappa_t^{-(1-\epsilon)/2}$ for all $t \ge t_1$,
which, along with Lemma~\ref{Lem:rudiment}, gives
\begin{align}
\lim_{t\rightarrow\infty} \left \vert \Phi(u_n(z),t) - \frac1{\wX(t)}
\int_{x_0}^z H(y,t)dy\right \vert = 0,
\label{Eq:psilim}
\end{align}
where $z$ should be regarded as a certain monotonically increasing function of $t$ with $x_0 < z \le t$.%
\medskip%

Now consider the other case.
Fix $0<\epsilon<1$ and define
$x_\pm := x_c \pm  \kappa_t^{-(1+ \epsilon)/2}$ and also
$z_\pm := x_c \pm 2  \kappa_t^{-(1+ \epsilon)/2}$.
By \eqref{Eq:wcon}, we always have
$
h(x,t) 
\le h(x_\pm,t) + \xi_\pm (x-x_\pm)
$,
where
\begin{align*}
\xi_\pm = \left . \frac{\partial h}{\partial x}\right \vert_{x=x_\pm}
=- \frac1\alpha \logn{x_\pm} - \log \omega_1\left (x_\pm\right )  + (t-x_\pm) 
L_1(x_\pm) \left ( \frac1\alpha + \Omega_1(x_\pm) \right ).
\end{align*}
Since $\lim_{t\rightarrow\infty} \kappa_t^{-(1+\epsilon)/2}/x_c = 0$, Taylor's theorem gives
$$
\xi_\pm  \sim \pm \left (\left . \frac{\partial^2h}{\partial x^2}\right \vert_{x=x_c} \right )  \kappa_t^{-(1+\epsilon)/2} = \mp \kappa_t^{(1-\epsilon)/2},\quad
\xi_\pm (z_\pm - x_\pm )  \sim - \kappa_t^{-\epsilon}.
$$
Now consider
\begin{align}
\nonumber
I_1(t):=&\int_{x_0}^{z_-} \frac{H(y,t)}{H(x_c,t)}dy\le \int_{x_0}^{z_-} e^{h(y,t)-h(x_-,t)} dy\\
\nonumber
&\le \int_{-\infty}^{z_-} e^{\xi_- (y-x_-)} d y
\sim \kappa_t^{-(1-\epsilon)/2} \exp \left (- \kappa_t^{-\epsilon} \right ),\\
\nonumber
I_2(t):=&\int_{z_+}^t \frac{H(y,t)}{H(x_c,t)}dy
\le \int_{z_+}^{t} e^{h(y,t)-h(x_+,t)} dy\\
&\le \int_{z_+}^{\infty} e^{\xi_+ (y-x_+)} d y
\sim \kappa_t^{-(1-\epsilon)/2} \exp\left (- \kappa_t^{-\epsilon} \right ),
\label{Eq:I12}
\end{align}
where we have used $H(x_\pm,t) \le H(x_c,t)$.
Since $\wX(t) \ge H(x_c,t)$ for all sufficiently large $t$ and
$
\lim\limits_{t\rightarrow\infty}I_1(t) = 
\lim\limits_{t\rightarrow\infty}I_2(t) = 0,
$
\eqref{Eq:psilim} yields, for any $\epsilon>0$,
$$
\lim_{t\rightarrow\infty} \Phi(u_n(z),t) = 
\begin{cases} 
0,& z \le x_c - \kappa_t^{-(1+\epsilon)/2},\\
1, & z \ge x_c + \kappa_t^{-(1+\epsilon)/2}.
\end{cases}
$$
Hence, it is enough to consider $\Phi(u_n(z),t)$ for 
$\vert x_c - z \vert \le \kappa_t^{-(1+\epsilon)/2}$ for a certain positive $\epsilon$.\medskip

Abbreviate $z:=x_c + y/\sqrt{\kappa_t}$ and assume $\vert y\vert  \le \kappa_t^{-1/8}$ 
(in a sense, we have set $\epsilon =1/4$).
By Taylor's theorem, there is
$y_0$ such that
$\vert y_0 \vert\le \vert y \vert$
and
\begin{align*}
h\left (z,t\right )= h(x_c,t) - \frac{1}2 y^2 + R_t\left (x_c+\frac{y_0}{\sqrt{\kappa_t}} \right ) y^3,\quad
R_t(x):= \frac{t}{6\alpha}L_3(x) \phi_2(x,t).
\end{align*}
Defining 
$$
\epsilon_1(t) = \exp\left (\sup\left \{ \left \vert R_t
\left (x_c + \frac{y_0}{\sqrt{\kappa_t}}\right ) y^3 \right \vert : 
\vert y\vert \le \kappa_t^{-1/8} \right \}\right ) -1,
$$
we have
\begin{align}
\label{Eq:Happrox}
\frac{H(z,t)}{H(x_c,t)} \simeq_{\epsilon_1(t)} \exp \left ( -\frac{y^2}{2} \right ),
\end{align}
where $A\simeq_{\epsilon} B$ is a shorthand notation for $(1-\epsilon)B
\le A \le (1+\epsilon)B$.
Then,
\begin{align}
\label{Eq:Xdt}
\int_{x_c - a^{-5/8}_t}^z \frac{H(x,t)}{H(x_c,t)} dx\simeq_{\epsilon_1(t)} \kappa_t^{-1/2}\int_{-\kappa_t^{-1/8}}^y \exp \left ( -\frac{x^2}{2} \right )dx,
\end{align}
where $\kappa_t^{-1/2}$ is the Jacobian of the change of variables.
Since $R_t \sim d_t \sim \kappa_t/(3 x_c)$ and, accordingly, $\lim_{t\rightarrow\infty} \epsilon_1(t) = 0$,
we have
\begin{align}
\label{Eq:XHt}
\lim_{t\rightarrow\infty} \frac{\wX(t)\sqrt{\kappa_t}}{ H(x_c,t)} = 
\lim_{t\rightarrow\infty} \int_{-\kappa_t^{-1/8}}^{\kappa_t^{-1/8}} e^{-x^2/2}dx=
\int_{-\infty}^\infty e^{-x^2/2} dx = \sqrt{2\pi},
\end{align}
which, together with \eqref{Eq:psilim}, gives
\begin{align}
\label{Eq:Psilim}
\lim_{t\rightarrow\infty} \Phi(u_n(x_c + y /\sqrt{\kappa_t}),t)
= \Upsilon(y).
\end{align}
To complete the proof, we have to show
$$
\lim_{t\rightarrow\infty} \Phi\left (u_n\left (x_c + \frac{y}{\sqrt{\kappa_t}}\right ),t\right )=
\lim_{t\rightarrow\infty} \Phi(\wS_t + \ds_t y,t),
$$
for $\vert y\vert  \le \kappa_t^{-1/8}$.
Let $S_t':=u_n(x_c)$ and let $y_c$ be a function of $t$ implicitly defined as the solution of the equation
$$
\left . \frac{\partial h_2(x,t)}{\partial x} \right \vert_{x=y_c}
= L_1(y_c)\left (\nu + \Omega_2^{(1)}(y_c)\right )
+ \left . \frac{\partial h(x,t)}{\partial x} \right \vert_{x=y_c},
$$
where $h_2(x,t) := \log u_n(x) + h(x,t) = \log (u_n(x) H(x,t) )$.
Notice that $u_n(x_c) \sim v_n(t)$.
Obviously, $y_c \sim x_c$.  Define
$$
\rho_1(t):=\frac{\wS_t}{S_t'}
= \frac{1}{\wX(t) S_t'} \sum_{k=x_0}^t u_n(k) \mkd(t)
= \frac{1}{\wX(t) S_t'} \sum_{k=x_0}^t e^{h_2(k,t)}.
$$
Since $h_2(x,t)$ for given $t$ satisfies the condition in Lemma~\ref{Lem:rudiment},
we have
$$
\left \vert \rho_1(t) - \frac{1}{\wX(t) S_t'} \int_{x_0}^t u_n(y) H(y,t) dy \right \vert 
\le  \frac{7H(y_c,t) u_n(y_c)}{\wX(t) S_t'}.
$$
Since $\lim_{t\rightarrow\infty}H(y_c,t) / \wX(t) =0$
and $u_n(y_c)/S_t' \sim 1$, we have
$$
\lim_{t\rightarrow\infty} \left \vert \rho_1(t) 
-\frac{1}{\wX(t) S_t'} \int_{x_0}^t u_n(y) H(y,t) dy \right \vert = 0.
$$

Let $z_\pm = x_c \pm \kappa_t^{-5/8}$.
Since
\begin{align*}
\int_{z_+}^t u_n(y)^m \frac{H(y,t)}{H(x_c,t)} dy 
\le 
u_n(t)^m \int_{z_+}^t \frac{H(y,t)}{H(x_c,t)} dy, \\
\int_{x_0}^{z_-} u_n(y)^m \frac{H(y,t)}{H(x_c,t)} dy 
\le u_n(t)^m
\int_{x_0}^{z_-} \frac{H(y,t)}{H(x_c,t)} dy,
\end{align*}
$I_1$ and $I_2$ in \eqref{Eq:I12} with $\epsilon=1/4$ yield
$$
\lim_{t\rightarrow\infty} \frac{1}{\wX(t) S_t'}\left \vert  \int_{x_0}^t u_n(y)^m H(y,t) dy 
- \int_{z_-}^{z_+} u_n(y)^m H(y,t) dy \right \vert = 0,
$$
for any positive integer $m$.
Using \eqref{Eq:Happrox}, we have
$$
\frac{1}{\wX(t) S_t'}\int_{z_-}^{z_+} u_n(z) H(z,t) dz
\simeq_{\epsilon_1(t)} \frac{H(x_c,t)\kappa_t^{-1/2}}{\wX(t)S_t'}
\int_{-\kappa_t^{-1/8}}^{\kappa_t^{-1/8}} u_n\left (x_c + \frac{y}{\sqrt{\kappa_t}}\right )
e^{-y^2/2} dy.
$$
Since $S_t' \sim u_n(x_c + y/\sqrt{\kappa_t})$, we have
$$
\lim_{t\rightarrow\infty}
\frac{1}{\wX(t) S_t'}\int_{z_-}^{z_+} u_n(y) H(y,t) dy = 1.
$$
Therefore $\rho_1(t) \sim 1$ or 
$\wS_t \sim u_n(x_c) \sim v_n(t)$,
as claimed.\medskip

Define
\begin{align*}
\sigma_t' &:= \kappa_t^{-1/2}\left . \frac{du_n}{dx}\right \vert_{x=x_c} 
=\frac{S_t'}{\sqrt{\kappa_t}}L_1(x_c)\left [ \frac{1}{\alpha} + \Omega_1(x_c) \right ],\\
\rho_2(t) &:= \frac{\wS_t - S_t'}{\sigma_t'}
= \frac{1}{\wX(t) } \sum_{k=x_0}^t \frac{u_n(k) - u_n(x_c)}{\sigma_t'} H(k,t),\\
\rho_3(t) &:= \frac{1}{\wX(t) }\int_{z_-}^{z_+}\frac{u_n(x) - u_n(x_c)}{\sigma_t'} H(x,t) dx.
\end{align*}
Note that $\sigma_t' \sim \s_n(t)$.
%
Assume $\vert y\vert  \le \kappa_t^{-1/8}$. By Taylor's theorem, there is $y_1$ with
$\vert y_1 \vert \le \vert y \vert$ such that
$$
\frac{u_n(x_c + y /\sqrt{\kappa_t}) - u_n(x_c)}{\sigma_t'} = y + \frac{\tilde R_t(x_c+y_1/\sqrt{\kappa_t} )}{\sigma_t'} y^2,
$$
where
$$
\tilde R_t(x):= \frac{1}{2 \kappa_t} \frac{d^2 u_n(x)}{dx^2}=\frac{u_n(x)}{2 \kappa_t}L_2(x)
\left (
\frac{L_1^2}{L_2}\left [ \left ( \frac{1}{\alpha} +\Omega_1(x) \right )^2 -\Omega_2(x) 
\right ]
-\frac{1}{\alpha} -\Omega_1(x) 
\right ) .
$$
Using
\begin{align}
\frac{\tilde R_t(x_c+y_1/\sqrt{\kappa_t})}{\sigma_t'}\sim\frac{\tilde R_t(x_c)}{\sigma_t'} \sim 
\frac{1}{2\sqrt{\kappa_t x_c^2}} \left ( \frac{\delta_{n,1}}{\alpha}-1 \right )
\label{Eq:Rts}
\end{align}
for $\vert y_1\vert  \le \kappa_t^{-1/8}$, $\int_{-x}^x y e^{-y^2/2} dy = 0$,
and \eqref{Eq:Happrox}, we have
$$
\vert \rho_3(t) \vert \simeq_{\epsilon(t)}
\frac{\kappa_t^{-1/2} H(x_c,t)}{\wX(t)}
\frac{1}{2\sqrt{\kappa_t x_c^2}} \left \vert \frac{\delta_{n,1}}{\alpha}-1 \right \vert
\int_{-\kappa_t^{-1/8}}^{\kappa_t^{-1/8}}
y^2 e^{-y^2/2} dy,
$$
where $\lim_{t\rightarrow\infty}\epsilon(t)= 0$. Therefore,
\begin{align}
\label{Eq:rho3}
\lim_{t\rightarrow\infty} \rho_3(t) = 0.
\end{align}
Since
\begin{align*}
\left \vert \frac{1}{\wX(t) }\int_{x_0}^{z_-}\frac{u_n(x) - u_n(x_c)}{\sigma_t'} H(x,t) dx\right \vert
\le \frac{2u_n(t)}{\sigma_t'} 
\left \vert \frac{1}{\wX(t) }\int_{x_0}^{z_-} H(x,t) dx\right \vert,\\
\left \vert \frac{1}{\wX(t) }\int_{z_+}^{t}\frac{u_n(x) - u_n(x_c)}{\sigma_t'} H(x,t) dx\right \vert
\le \frac{2u_n(t)}{\sigma_t'} 
\left \vert \frac{1}{\wX(t) }\int_{z_+}^{t} H(x,t) dx\right \vert,
\end{align*}
\eqref{Eq:I12} together with \eqref{Eq:rho3} gives
\begin{align}
\label{Eq:rho2}
\lim_{t\rightarrow\infty}\rho_2(t)=
\lim_{t\rightarrow\infty}\frac{1}{\wX(t) }\int_{z_-}^{z_+}\frac{u_n(x) - u_n(x_c)}{\sigma_t'} H(x,t) dx=0.
\end{align}

Define
\begin{align*}
	\rho_4(t):=&\frac{{\ds_t}^2}{(\sigma_t')^2} = 
\frac{1}{\wX(t)(\sigma_t')^2} \sum_k \left (u_n(k) - S_t' - \sigma_t'\rho_2(t) \right )^2 H(k,t)
\\
=& \frac{1}{\wX(t)} \sum_k \left (\frac{u_n(k) - u_n(x_c)}{\sigma_t'} \right )^2 H(k,t) - \rho_2(t)^2,\\
\rho_5(t):=&\frac{1}{\wX(t)} \int_{z_-}^{z_+} \left (\frac{u_n(x) - u_n(x_c)}{\sigma_t'} \right )^2 H(x,t) dx\\
=&\frac{1}{\wX(t)} \int_{z_-}^{z_+} \kappa_t (x-x_c)^2 \left ( 1 + \frac{\tilde R_t(x_c+y_1/\sqrt{\kappa_t})}{\sigma_t'} \right )^2 H(x,t) dx.
\end{align*}
Using~\eqref{Eq:Happrox}, \eqref{Eq:XHt}, \eqref{Eq:Rts}, and
\eqref{Eq:rho2}, we have
$$
\lim_{t\rightarrow\infty} \rho_4(t) =
\lim_{t\rightarrow\infty} \rho_5(t) = \frac{1}{\sqrt{2\pi}} \int_{-\infty}^\infty y^2 e^{-y^2/2} dy
=1,
$$
where we have also used the same procedure to arrive at \eqref{Eq:rho2} 
using $(u_n(x) - u_n(x_c))^2 \le 4u_n(t)^2$.
From the above calculations, we conclude that there is a constant $C$ such that
\begin{align}
\vert \rho_2(t) \vert \le \frac{C}{\sqrt{\kappa_t} x_c},\quad
\vert \rho_4(t) -1\vert \le \frac{C}{\sqrt{\kappa_t} x_c},
\label{Eq:r24}
\end{align}
for all sufficiently large $t$.\medskip

Let $z :=x_c + y/\sqrt{\kappa_t}$ and $z' := u_n^{-1} ( \wS_t + \ds_t y  )$.
Recall that for any small but positive $\epsilon_2$ and $\epsilon_3$, 
$\wX(t) \ge \kappa_t^{-(1-\epsilon_2)/2} H(x_c,t)$ 
and $\kappa_t \le t^{-1+\epsilon_3}$ for all sufficiently large $t$. Since
\begin{align*}
\lim_{t\rightarrow\infty}\left \vert  \Phi(u_n(z),t) - \Phi(u_n(z'),t) \right \vert
&= 
\lim_{t\rightarrow\infty}\frac{1}{\wX(t)}
\left \vert \int_z^{z'} H(x,t) dx \right \vert
\le \lim_{t\rightarrow\infty} t^{-(1-\epsilon_0)/2}\vert z - z' \vert,
\end{align*}
for any $0<\epsilon_0<1$,
we need to show that there is $\epsilon_0$ such that
$\lim_{t\rightarrow\infty} t^{-(1-\epsilon_0)/2} \vert z - z' \vert =0$.
First observe that
$
\wS_t + \ds_t y 
= S_t' + \sigma_t' y'$ for
$y':=\rho_2(t) +  y \sqrt{\rho_4(t)}.
$
Assume $t$ is so large that $\vert y' \vert \le 2 \kappa_t^{-1/8}$.
By Taylor's theorem, there is
$y_1$ such that
$\vert y_1 \vert\le \vert y' \vert \le 2 \kappa_t^{-1/8}$ and
\begin{align*}
z' &= u_n^{-1}(S_t') + \frac{\sigma_t'}{u_n'(z_1)} y'\\
&= z+ \left ( \frac{\sigma_t'}{u_n'(z_1)} -\frac{1}{\sqrt{\kappa_t}}\right )
y+\frac{\sigma_t'}{u_n'(z_1)} \left [\rho_2(t) + y (\sqrt{\rho_4}-1 )\right ],
\end{align*}
where $z_1= u_n^{-1}(S_t' + \sigma_t' y_1)$. 
Using $z_1 \sim x_c$, $u_n(x_c) =\sigma_t'\sqrt{\kappa_t}$, \eqref{Eq:r24}, and
$\lim_{t\rightarrow\infty} t^{-\epsilon_4}/(\kappa_t x_c) =0$ for any $\epsilon_4>0$, we have
\smash{$
\vert z'-z\vert  \le \kappa_t^{-1/6} \le t^{1/4}
$}
for all sufficiently large $t$. Hence, if we choose $\epsilon_0=1/8$, 
we have the desired result. 
Since $\rho_4(t) \sim 1$, the proof is completed. 
\end{proof}
\begin{remark}
Since $u_n(t)/v_n(t) \sim (\log t)^{\delta_{n,1}/\alpha}$,
we have 
\begin{align*}
    \wS_t \sim W_t & \text{ $n \ge 2$, while }\\
    \lim_{t\rightarrow\infty} \wS_t/W_t = 0 & \text{ for $n=1$.}
\end{align*}
In other words, when $n \ge 2$, the mean fitness at generation $t$ is hardly discernible from the largest fitness at the same generation. Another interesting observation is that  if $n\ge 2$ or if $n=1$ and $\alpha>2$, then
$
\lim_{t\rightarrow\infty} \ds_t = 0,
$
which implies that the width of the traveling wave decreases to zero and
the \efd becomes a delta function 
in the sense that 
$$
\lim_{t\rightarrow\infty}\Phi(\wS_t + y ,t) =
\begin{cases}
0,& y<0,\\
1,& y>0.
\end{cases}
$$ 
This should be compared with the case of
$n=1$ and $\alpha<2$ in which the width of the traveling wave  increases with generation.
For $n=1$ and $\alpha=2$, the behaviour of $\ds_t$ depends on the slowly varying function $L$ entering the tail function in
\eqref{Eq:gasym}.
\end{remark}
\subsection{\label{Sec:sfm}Semi-deterministic \fm and its \efd}

\noindent
{\bf Definition of the \sfm.}
At each generation $k\ge 0$ a new mutant with fitness $$\theta_k:=(1-\beta) u_n(k)$$ appears and $(N_k(t) \colon t\ge k)$ are mutually independent Galton-Watson processes with Poisson-distributed offspring with mean $\theta_k$ for each $k$. In case $u_n(k)$ is ill-defined, we set $\theta_k=1$. 
By definition, $N_k(k)=1$ and $N_k(\tau) = 0$ for $\tau < k$
and no extinction is possible in the \sfm. 
Since we will use Lemma~\ref{Thm:efd} to prove Theorem~\ref{Thm:sfm} below, 
we limit the definition of the \sfm  to the case 
$n=1$ or the case $n=2$ and $\alpha<1$; see also Remark~\ref{Rem:um}.
\bigskip

We denote the total population size of the \sfm at generation $t$ by
$$
\Xs(t) := \sum_{k=0}^t N_k(t).
$$
The \efd $\Psi_s(f,t)$ of the \sfm and its mean fitness $S_t$ are defined as
\begin{align*}
\Psi_s(f,t) := \frac{1}{\Xs(t)}\sum_{k=0}^{t} N_k(t) \Theta(f-u_n(k)),\quad
\Ss_t := \frac{1}{\Xs(t)}\sum_{k=0}^{t} u_n(k)N_k(t).
\end{align*} 
Since $N_k(t)$ is the number of non-mutated descendants, we put $(1-\beta)$ in the definition of the fitness of a new mutant
in the \sfm.
In a sense, the \sfm is closer to the \fm than the \dfm 
due to fluctuations of $N_k(t)$.
We redefine $u_n(k) := \theta_k/(1-\beta)$ for convenience.
Now we prove that the \efd of the \sfm in the long time limit
becomes almost surely a Gaussian traveling wave just as the \dfm.

\begin{theorem}
\label{Thm:sfm}
For the \sfm with $n=1$ or with $n=2$ and $\alpha<1$,
almost surely 
\begin{align*}
\lim_{t\rightarrow\infty} \Psi_s(v_n(t) +  \s_n(t) y,t)= \Upsilon(y),\quad
\lim_{t\rightarrow\infty} \frac{\Ss_t}{v_n(t)}=1,
\end{align*}
where
\begin{align*}
v_n(t) &=
\alpha^{-\delta_{n,1}/\alpha} 
\left ( \lo{n-1}t \right )^{1/\alpha}
L \left (\left ( \lo{n-1}t \right )^{1/\alpha} \right )
,\\
\s_n(t) &= \frac{v_n(t)}{\sqrt{\alpha t}} \left ( \prod_{k=1}^{n-1} \lo{k}{t}
\right )^{-1/2} 
\end{align*}
have been introduced previously in \eqref{Eq:vnt} and \eqref{Eq:snt}.
\end{theorem}
\begin{proof}
We define $J$ and $E$ as in Lemma~\ref{Thm:efd}.
It is obvious that Lemma~\ref{Thm:efd} is applicable to the \sfm.
Note that by definition $J$ in \eqref{Eq:IJ} for the \sfm can be regarded as the sample space and, accordingly, $\P(E)=1$.
For any  $0<\epsilon<1/2$ and for any outcome $\omega \in E$, there exists $\tau_1$ such that
$(1-\epsilon)\theta_k^{t-k} \le N_k(t) \le ( 1 + \epsilon) \theta_k^{t-k}$
for all $t \ge k \ge \tau_1$.
Notice that $\tau_1$ can vary from outcome to outcome.
Let 
$$
\Xs(t,\tau_1) := \sum_{k=0}^{\tau_1} N_k(t),\quad
\Xd(t) := \sum_{k=0}^t \theta_k^{t-k},\quad
\Xd(t,\tau_1) := \sum_{k=0}^{\tau_1} \theta_k^{t-k}.
$$
Then, for $t \ge \tau_1$, we have
$$
(1-\epsilon) \left (\Xd(t) - \Xd(t,\tau_1) \right )+ \Xs(t,\tau_1)
\le
\Xs(t) 
\le
(1+\epsilon) \left (\Xd(t) - \Xd(t,\tau_1) \right )+ \Xs(t,\tau_1).
$$
Since $\Xd_s(t,\tau_1)$ and $\Xd(t,\tau_1)$ grow at most exponentially and
$\Xd(t)$ grows super-exponentially,
we have almost surely
$$
\liminf_{t\rightarrow\infty} \frac{\Xs(t)}{\Xd(t)} \ge 1-\epsilon,\quad
\limsup_{t\rightarrow\infty} \frac{\Xs(t)}{\Xd(t)} \le 1+\epsilon.
$$
Hence there is almost surely $\tau_2$ such that
$(1-2\epsilon) \Xd(t) \le \Xs(t) \le (1+2\epsilon) \Xd(t)$ for all $t \ge \tau_2$.\medskip

Now set $\tau = \max\{\tau_1,\tau_2\}$ and assume $t > \tau$.
Then, we have
\begin{align*}
\Psi_s(f,t) &\ge \frac{1}{1+2\epsilon} \frac{1}{\Xd(t)} \sum_{k=0}^{\tau_1} N_k(t)\Theta(f-u_n(k))
 +\frac{1}{1+2\epsilon} \frac{1}{\Xd(t)} \sum_{k=\tau_1+1}^{t} N_k(t)\Theta(f-u_n(k))\\
&\ge \frac{1}{1+2\epsilon} \frac{1}{\Xd(t)} \sum_{k=0}^{\tau_1} N_k(t)\Theta(f-u_n(k))
+\frac{1-\epsilon}{1+2\epsilon} \frac{1}{\Xd(t)} \sum_{k=\tau_1+1}^{t} \theta_k^{t-k}\Theta(f-u_n(k)).
\end{align*}
Hence by Lemma~\ref{Lem:deter}, we conclude
\begin{align*}
\liminf_{t\rightarrow \infty} \Psi_s(v_n(t) + \s_n(t) y,t) \ge
\frac{1-\epsilon}{1+2\epsilon} \Upsilon(y).
\end{align*}
By the same token, we have
\begin{align*}
\limsup_{t\rightarrow \infty} \Psi_s(v_n(t) + \s_n(t) y,t) \le
\frac{1+\epsilon}{1-2\epsilon} \Upsilon(y).
\end{align*}
Since $\epsilon$ is arbitrary, we proved the first part of the theorem.\medskip

Let 
$$
\wS_t:= \frac{1}{\Xd(t)}\sum_{k=0}^t u_n(k) \theta_k^{t-k}.
$$
By Lemma~\ref{Lem:deter}, we have $\wS_t \sim v_n(t)$.
Inspecting the above proof, we can conclude that
for any $0<\epsilon<1/2$ and for any outcome $\omega\in E$, there is $\tau$ such that
$$
(1-\epsilon)\wS_t \le \Ss_t \le (1+\epsilon) \wS_t,
$$
for all $t \ge \tau$.
Since $\epsilon$ is arbitrary,
the proof is completed.
\end{proof}
\subsection{\label{Sec:num}Numerical study for the \mm with $n=1$}
Since the largest fitness is expected to dominate the evolution of the
population even in the \mm and the limiting distribution is continuous
even in the \fm, 
we conjecture that Theorem~\ref{Th:efd} is valid even for the \mm; see Remark~\ref{Rem:mm}. 
For the \mm, however, we only present some numerical results, which
supports our conjecture.\medskip

\begin{figure}[t]
\centering
\includegraphics[width=0.7\linewidth]{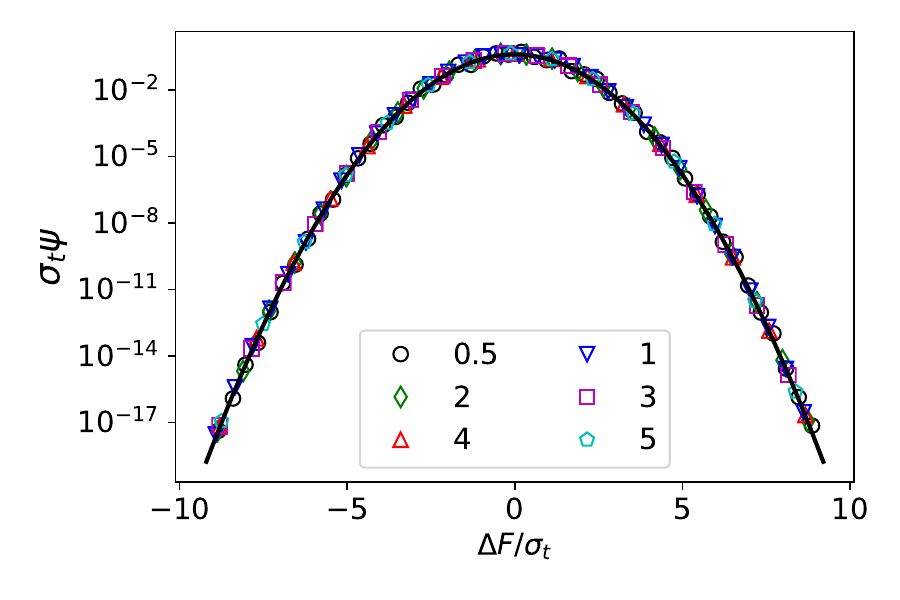}
\caption{\label{Fig:gumtra} Semilogarithmic plot of
$\sigma_t \psi(F,t)$ vs. $\Delta F/\sigma_t$
for various $\alpha$'s at $t=983~040$. For comparison, the normal distribution
is plotted by a solid curve.
}
\end{figure}
For numerical feasibility we assume that the fitness of a mutant can only be one of the 
discrete values
$f_i = (c i)^{1/\alpha}$ for $i\ge 1$, 
where $c$ is a constant to be determined later.
Defining $G_p(x) = \exp(-x^\alpha + c)$ for $x \ge c^{1/\alpha}$
and $G_p(x)=1$ for $x \le c^{1/\alpha}$, we assign
probabilities
\begin{align*}
p_i:=\P(F=f_i) = G_p(f_i) - G_p(f_{i+1})=e^{-c i} (e^c - 1).
\end{align*}
Since $G_p(f_{i+1}) \le G(x) \le G_p(f_i)$ for $f_i \le x < f_{i+1}$ and $\lim_{i\rightarrow\infty}
f_{i+1}/f_i = 1$, we have 
\begin{align*}
\lim_{x\rightarrow\infty} \frac{\log G(x)}{x^\alpha} = -1.
\end{align*}
Therefore, we can apply Theorem~\ref{Th:mainthm}, to predict 
$
W_t\sim \alpha^{-1/\alpha} (t \log t)^{1/\alpha},
$
almost surely on survival.\medskip

In this section, we denote the number of individuals with fitness 
$f_k$ at generation $t$ by $N_k(t)$. We would like to emphasize that 
$f_k$ should not be confused with $W_k$.
The total population size $X(t)$ and the mean fitness $S_t$ are calculated as
\begin{align}
X(t) = \sum_{k=1}^\infty N_k(t) ,\quad
S_t = \sum_{k=1}^\infty \frac{N_k(t)}{X(t)} f_k.
\end{align}
The standard deviation $\sigma_t$ is naturally defined.
Given $N_k(t)$ and $S_t$, the random variable
$N_k(t+1)$ is drawn from the Poisson distribution with mean $(1-\beta)N_k(t) f_k + \beta S_t X(t)
p_k$. Since the accurate value of $\beta$ is not important
as long as $0 < \beta < 1$, we choose $\beta = 10^{-20}$
to make $1-\beta$ indistinguishable from $1$ within machine accuracy
of double-precision floating-point format.\medskip

\pagebreak[3]

Since the total size of the population increases super-exponentially on survival
and we are mostly interested in long-time behaviour, 
we set $X(0)$ very large (in the actual implementation, we set
 $ N_1(0)=X(0) = 10^{100}$ and $S_0=f_1$),
which makes fluctuations of the total population size invisible within machine accuracy.
Besides, we set $c=20\log 10 \approx 46.05$, which gives
$p_{k+1}/p_k = 10^{-20}$. Therefore, we have only to consider
$k$ up to $\beta S_t X(t) p_k \ge 1$ with $p_k \approx e^{-c(k-1)}$.\medskip

Let $\psi_k(t):= N_k(t)/X(t)$.
Since parameters are chosen such that deviation from the expected 
value of $\psi_k(t+1)$ for given $\psi_k(t)$ cannot be generated
within machine accuracy, the actual stochastic simulations cannot be 
different from the deterministic equation 
\begin{align}
\psi_{k}(t+1) = {(1-\beta)} \psi_k(t) \frac{f_k}{S_t} + \beta \tilde p_k,\quad
{S_t} = \sum_k \psi_k(t) f_k,
\label{Eq:fre}
\end{align}
where $\tilde p_k = p_k$ if $\beta X(t+1)p_k >1$ and 0, otherwise.
In a sense, we are studying a deterministic version of the \mm, but,
as we mentioned already, even the full stochastic \mm 
is not distinguishable from the deterministic version \mm for the parameters
we chose.
Now, we present the numerical solution of \eqref{Eq:fre}.
\medskip

In Figure~\ref{Fig:gumtra}, we depict $\sigma_t \psi(F,t)$ vs $\Delta F/\sigma_t$,
where $\Delta F = F - S_t$ on a semi-logarithmic scale
at generation $t \approx 10^6$.
Here, $\psi(F,t)$ is a density that is calculated as
$$
\psi(F,t) = \frac{1}{f_{k+j}-f_{k-j}} \sum_{k-j\le i\le k+j} \psi_i(t)
$$
with a suitable bin size $2j$, where the integer $k$ is determined uniquely by
$f_{k} \le F < f_{k+1}$. We assure that dependency of $\psi(F,t)$ on the bin size
is negligible over a wide range of $j$ (details not shown here).
For comparison, the Gaussian function with zero mean and unit variance
is also drawn by a solid curve. Just  as we proved for the \fm,
the \efd is again well described by a Gaussian traveling wave.
\medskip

\begin{figure}[t]
\centering
\includegraphics[width=0.95\linewidth]{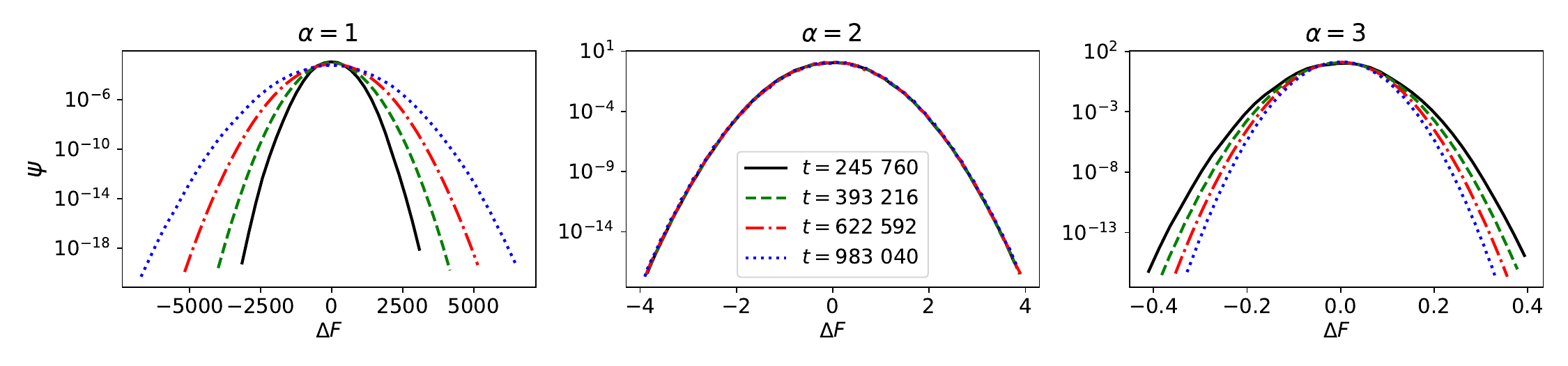}
\caption{\label{Fig:gumtradT} Plots of $\psi(F,t)$ vs. $\Delta F$ at different
generations for $\alpha=1$ (left), $\alpha=2$ (middle), and $\alpha=3$ (right) on a semi-logarithmic scale.
For $\alpha=3$ ($\alpha=1$), the width of the traveling wave decreases (increases).
For $\alpha=2$, the width of the traveling wave remains constant.
}
\end{figure}

We have found that depending on the actual form of the tail function,
$\sigma_t$ can increase, decrease,
or even remain constant in the \fm. To check if this property remains valid in \mm, we plotted the \efd at different times for different values of $\alpha$, whose result is summarized in Figure~\ref{Fig:gumtradT}.
The behaviour is the same as shown for the \fm. 
In fact, the predicted $S_t$ and $\sigma_t$ for the \fm conform to numerical results (details not shown here).
From the numerical observations, we conjecture that {the travelling-wave part of} the \mm with type I tail function 
(at least with $n=1$) has the same \efd as the \fm.

\section{\label{Sec:sum}Concluding remarks}

We provided strong analytical and numerical evidence for the emergence of a travelling wave for the branching process with selection and mutation for unbounded fitness distributions of Gumbel type. For type~I tail functions with tail index $n=1$, or in other words stretched exponential fitness distributions, we show that if the tail parameter satisfies $\alpha > 2$, the standard deviation of the traveling Gaussian
wave decreases and eventually the \efd becomes highly peaked like a delta function. 
{Traveling wave solutions of 
Gaussian form were found previously in a study of the deterministic (infinite population) limit of the model, which amounts to solving the recursion (\ref{Eq:fre}) with $\tilde p_k = p_k$, see~\cite{Park2008}. The expressions for the mean and variance of the EFD obtained in \cite{Park2008} for a particular type~I tail function match Eqs.~\eqref{Eq:vnt} and~\eqref{Eq:snt}, see also~\cite{Kingman1978}.}  
\medskip

{We conjecture a similar behaviour} for bounded fitness distributions of Gumbel type in the condensation case discussed in Section~1. In that case the Gaussian wave is expected to travel 
to the essential supremum of the fitness distribution, while its standard deviation goes to zero faster than the distance of its mean to the essential supremum. 
For bounded fitness distributions of Weibull type we conjecture, as in Ref.~\cite{Dereich2017} for a branching model in continuous time, that the condensate emerges in the shape of a Gamma distribution. The conjecture is justified by the rigorous analysis of {the} deterministic model in Ref.~\cite{Dereich2013}.
\medskip

In our model every individual has a Poisson number of offspring with mean given by its fitness. It is natural to conjecture that results like emergence of the travelling wave, doubly exponential growth rates or condensation also hold for other distributions with the same mean and not too large variance. Verifying this universality conjecture rigorously would be an interesting future project.
\medskip

\bibliographystyle{unsrt}
\bibliography{f}
\end{document}